\documentclass[lettersize,journal]{IEEEtran}
\usepackage{graphicx,subcaption}
\usepackage{amsmath,amsfonts}
\usepackage{amsthm}
\usepackage{algorithmic}
\usepackage{algorithm}
\usepackage{array}
\usepackage{textcomp}
\usepackage{stfloats}
\usepackage{url}
\usepackage{verbatim}
\usepackage{graphicx}
\usepackage{cite}
\usepackage{multicol}
\usepackage{titling}
\usepackage{enumitem}

\DeclareMathOperator{\prox}{prox}
\newcommand{\norm}[1]{\Vert #1 \Vert}
\newcommand{\card}[1]{\vert #1 \vert}
\usepackage{stmaryrd}
\newcommand{\intervalle}[2]{\llbracket #1, #2 \rrbracket}
\usepackage{hyperref}
\usepackage{url}
\usepackage{multicol}
\usepackage{cleveref}
\newcommand{\appref}[1]{Appendix~\ref{#1}}
\newcommand{\suppref}[1]{Supplementary Material}
\usepackage{tcolorbox}
\tcbuselibrary{skins}
\usepackage[normalem]{ulem}

\usepackage{amsmath,amsfonts,bm}

\def\rvb{{\mathbf{b}}}

\def\rvu{{\mathbf{i}}}

\def\rvp{{\mathbf{p}}}

\def\rvu{{\mathbf{u}}}
\def\rvv{{\mathbf{v}}}

\def\rvx{{\mathbf{x}}}
\def\rvy{{\mathbf{y}}}
\def\rvz{{\mathbf{z}}}

\def\rmA{{\mathbf{A}}}

\def\vv{{\bm{v}}}

\def\mA{{\bm{A}}}

\DeclareMathAlphabet{\mathsfit}{\encodingdefault}{\sfdefault}{m}{sl}
\SetMathAlphabet{\mathsfit}{bold}{\encodingdefault}{\sfdefault}{bx}{n}

\def\gK{{\mathcal{K}}}
\def\gL{{\mathcal{L}}}

\def\gS{{\mathcal{S}}}

\def\sR{{\mathbb{R}}}

\newcommand{\R}{\mathbb{R}}

\DeclareMathOperator*{\argmin}{arg\,min}

\DeclareMathOperator{\sign}{sign}

 \usepackage[textsize=tiny]{todonotes}

\theoremstyle{definition}
\newtheorem{definition}{Definition}[section]
\theoremstyle{plain}
\newtheorem{theorem}{Theorem}[section]
\theoremstyle{plain}
\newtheorem{proposition}{Proposition}[section]
\newtheorem{lemma}{Lemma}[section]
\newtheorem{hypothesis}{Hypothesis}[section]
\theoremstyle{remark}
\newtheorem{remark}{Remark}[section]
\theoremstyle{definition}
\newtheorem{example}{Example}[section]

\begin{document}

\title{Proximal Operators of Sorted Nonconvex Penalties}

\author{Anne Gagneux, Mathurin Massias, and Emmanuel Soubies
\thanks{AG and MM are with Inria, ENS de Lyon, CNRS, Université Claude Bernard Lyon 1, LIP, 69342,
Lyon Cedex 07, France (e-mails: anne.gagneux@ens-lyon.fr, mathurin.massias@inria.fr). ES is with IRIT, Université de Toulouse, CNRS, 31000 Toulouse, France (e-mail: emmanuel.soubies@irit.fr). This work was supported by the ANR EROSION (ANR-22-CE48-0004) and the RT IASIS
through the PROSSIMO grant.}}

\IEEEpubid{}

\maketitle

\begin{abstract}

  This work studies the problem of sparse signal recovery with automatic grouping of variables.
  To this end, we investigate sorted nonsmooth penalties as a regularization approach for generalized linear models.
  We focus on a family of \emph{sorted nonconvex  penalties} which generalizes the Sorted $\ell_1$ Norm (SLOPE).
  These penalties are designed to promote clustering of variables due to their sorted nature, while the nonconvexity reduces the shrinkage of coefficients.
  Our goal is to provide efficient ways to compute their proximal operator, enabling the use of popular proximal algorithms to solve composite optimization problems with this choice of sorted penalties.
  We distinguish between two classes of problems: \emph{the weakly convex case} where computing the proximal operator remains a convex problem, and \emph{the nonconvex case} where computing the proximal operator becomes a challenging nonconvex combinatorial problem.
  For the weakly convex case (\emph{e.g.} sorted MCP and SCAD), we explain how the Pool Adjacent Violators (PAV) algorithm can exactly compute the proximal operator.
  For the nonconvex case (\emph{e.g.} sorted $\ell_q$ with $q \in ]0,1[$), we show that a slight modification of this algorithm turns out to be remarkably efficient to tackle the computation of the proximal operator. We also present new theoretical insights on the minimizers of the nonconvex proximal problem.
  We demonstrate the practical interest of using such penalties on several experiments.
\end{abstract}


\section{Introduction}
\label{section:introduction}

Signal recovery can be cast as the following variational problem
\begin{equation}
    \label{pb:composite}
    \argmin_{\rvx \in \sR^p} g(\rvx) + \varPsi (\rvx) \, ,
\end{equation}
where $g : \sR^p \to \sR$ is a data-fidelity term and the \emph{penalty} $\varPsi$ is a regularization term that should embed some properties of the solution.
Among them, sparsity and structure are particularly useful for a model as they improve its interpretability and decrease its complexity.
Sparsity is most usually enforced through a penalty term favoring variable selection, i.e. solutions that use only a subset of features.
Such sparsity-promoting penalties share the common property to be nonsmooth at zero.
While being widely used, the convex $\ell_1$ norm (the Lasso \cite{tibshirani1996regression}) has the drawback of shrinking all nonzero coefficients towards 0. In the sparse recovery literature, this undesirable property is commonly called \emph{bias}.
It can be mitigated by using nonconvex penalties \cite{selesnick2014sparse}.
These include smoothly clipped absolute deviation (SCAD) \cite{fan2001variable}, minimax concave penalty (MCP)\cite{zhang2010mcp} or the log-sum penalty \cite{candes2008enhancing}.
Another very popular class of nonconvex penalties are the $\ell_q$ regularizers \cite{grasmair2008sparse}, with $q \in \left] 0, 1 \right[$.
The latter are not only nonconvex and nonsmooth (like MCP and SCAD) but also non-Lipschitz.
An alternative to using nonconvex penalties is \emph{refitting} \cite{belloni2013least,deledalle2017clear}.
This is a two-step approach where the Lasso is first used for \emph{support recovery only} and where the coefficients are then estimated solely on the identified support.
Limitations of such methods, including interpretability,  are studied in \cite{salmonrefitting}.
Here, we focus only on one-step methods where the penalty itself is designed to enforce all the desirable properties of the solution.

Beyond sparsity, one may be interested in clustering features, i.e. in assigning equal values to coefficients of correlated features.
In practice, such groups should be recovered \emph{without prior knowledge about the clusters}.
The Lasso is unfit for this purpose as it tends to only select one relevant feature from a group of correlated features.
The Fused Lasso \cite{tibshirani2005sparsity} is a remedy, but only works for features whose ordering is meaningful, as groups can only be composed of consecutive features.
Elastic-Net regression \cite{zou2003regression}, which combines $\ell_1$ and squared $\ell_2$ regularization, encourages a grouping effect, but does not enforce exact clustering.
In contrast, the OSCAR model \cite{bondell2008simultaneous} combines a $\ell_1$ and a pairwise $\ell_\infty$ penalty to simultaneously promote sparsity and equality of some coefficients, i.e. clustering.
OSCAR has been shown to be a special case of the sorted $\ell_1$ norm regularization \cite{bogdan2013,zeng2014ordered}.
As our present work heavily builds upon this penalty, we recall its definition.

\begin{definition}[SLOPE \cite{bogdan2013}, OWL \cite{zeng2014ordered}]
  \label{def:slope}
  The sorted $\ell_1$ penalty, denoted $\varPsi_{\mathrm{SLOPE}}$, writes, for $\rvx \in \sR^p$:
  \begin{equation}\label{eq:slope}
      \varPsi_{\mathrm{SLOPE}}(\rvx) = \sum_{i=1}^p \lambda_i |x_{(i)}| \, ,
  \end{equation}
  where $x_{(i)}$ is the $i$-th component of $\rvx$ sorted by non-increasing magnitude, i.e. such that $|x_{(1)}| \geq \dots \geq |x_{(p)}|$, and where $\lambda_1  \geq \dots \geq \lambda_p \geq 0$ are chosen regularization parameters.
\end{definition}

The main feature of SLOPE is its ability to exactly cluster coefficients \cite{figueiredo2014sparse,schneider2022geometry}.
Solving many issues of the Lasso, e.g. high False Discovery Rate \cite{su2017false}, SLOPE has attracted much attention in the last years.
It has found many applications both in compressed sensing \cite{el2019online}, in finance \cite{kremer2020sparse} or for neural network compression \cite{zhang2018learning}.
One line of research has focused on its statistical properties:  minimax rates \cite{su2016,bellec2018slope}, optimal design of the regularization sequence \cite{hu2019asymptotics,kos2020asymptotic}, pattern recovery guarantees \cite{bogdan2022pattern}.
Dedicated numerical algorithms were also developed, comprising screening rules \cite{larsson2020strong,elvira2023safe}, full path computation à la LARS \cite{tardivel2023solution} or hybrid coordinate descent methods \cite{larsson2023coordinate}.

Although it improves on Lasso regarding features clustering, SLOPE is convex and thus over-penalizes large coefficients.

The authors in~\cite{suzuki2024external} proposed to mitigate this drawback through the use of a specific penalty, whose proximal operator is computable as an external division of two SLOPE/OSCAR proximal operators.
Alternatively, it is tempting to use nonconvex variants of SLOPE, replacing the absolute value in \Cref{eq:slope} by a nonconvex sparse penalty.
However, one of the reasons behind the success of SLOPE is its practical usability through the availability of its proximal operator, which can be computed exactly \cite{zeng2014ordered,dupuis2022proximal}.
It comes down to solving an isotonic regression problem, which can be done efficiently using the Pool Adjacent Violators (PAV) algorithm~\cite{best1990active}.

For nonconvex sorted penalties, results are scarcer.
In \cite{feng2019sorted}, the authors studied the statistical properties of a class of sorted nonconvex penalties, including sorted MCP and sorted SCAD. They also proposed a ``majorization-minimization'' approach to deal with some sorted nonconvex penalties, using a linear convex approximation of the regularization term.

In this paper, we show how to compute the proximal operators of a wide class of nonconvex sorted penalties, including in particular sorted versions of SCAD, MCP, Log-Sum, $\ell_q$.
These penalties leverage both nonconvexity and sorting: as illustrated on \Cref{fig:denoising_solutions}, it is visible that, with equal cluster recovery performance, they lead to lower estimation error than SLOPE.
Appendix~\ref{app:details_clustering} provides more details on their clustering properties.
Being able to compute their proximal operator allows for their use in combination with any datafit, be it for regression or classification, as well as any proximal algorithm (e.g., FISTA, ADMM, primal-dual methods).

\begin{figure}[!t]
    \centering
    \includegraphics[width=\linewidth]{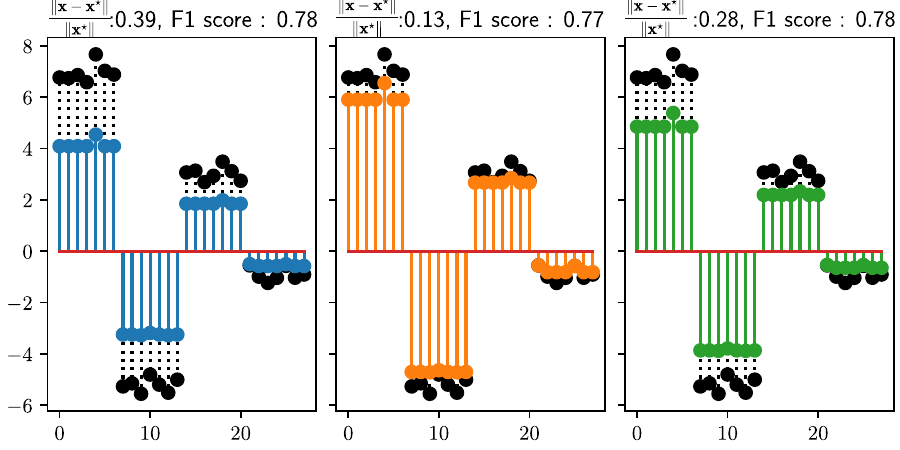}
    \caption{
        Sorted penalties for denoising: $\hat \rvx$ is the estimate obtained by applying the proximal operator of the penalty to a noisy version of the ground truth $\rvx^*$.
        With equivalent cluster recovery performance (same F1 score), sorted nonconvex penalties give amplitudes closer to the ground truth compared to convex SLOPE.
        Note that the clusters are contiguous only for visualization purposes: shuffling the noisy signal does not affect the results.
        Details on the setup and complete results are in \Cref{sub:denoising}.
        }
    \label{fig:denoising_solutions}
    \vspace{-1em}
\end{figure}

\paragraph*{Contributions and Outline} After recalling some general properties of proximal operators of sorted penalties in \Cref{section:background}, we present two novel contributions.

In \Cref{section:weakly-convex}, we consider \textbf{sorted weakly convex penalties}  (including sorted MCP, SCAD, and log-sum) in settings where computing the proximal operator comes down to a convex problem. We explain how it can be directly and exactly solved by the PAV algorithm.

In \Cref{section:nonconvex}, we consider a \textbf{general class of sorted nonconvex penalties} for which computing the proximal operator can result in a nonconvex optimization problem. This class notably includes
penalties that are not weakly convex, such as \textbf{sorted $\ell_q$ norms for $ q \in \left]0,1\right[$}.
Within this framework, we first conduct a general analysis of the variations of the scalar proximal objective function (\Cref{prop:variation_scalar}), which generalizes some peculiar results of the literature. We then establish necessary and sufficient conditions characterizing the local minimizers of the proximal problem (\Cref{theorem:ell_q_local}) and demonstrate that the PAV algorithm converges to such local minimizers (\Cref{thm:PAV_loc_cv}). Finally, we derive a necessary condition for global minimizers (\Cref{theorem:oscar_ell_q_global_nec}) and leverage it to propose a novel PAV-inspired algorithm, that we coin D-PAV (Decomposed-PAV).

Numerical experiments on denoising and regression problems are reported in \Cref{section:expe}, highlighting the benefits of sorted nonconvex penalties over SLOPE as well as the improvement of the proposed D-PAV over PAV.

\section{Background on Proximal Operator \\for Sorted Penalties}
\label{section:background}

In this section, we revisit key properties of sorted penalties, focusing on how the proximal problem simplifies to an isotonic problem.
We end this section by recalling the Pool Adjacent Violators (PAV) algorithm, which is the standard method for solving such isotonic problems.
We consider composite penalized problems that write as:
\begin{equation}
    \label{pb:composite}
    \argmin_{\rvx \in \sR^p} g(\rvx) + \varPsi (\rvx) \, ,
\end{equation}
where $g$ is a $L-$smooth convex datafit and the function $\varPsi$ is as follows.

\begin{hypothesis}
    \label{hyp:gen}
    Given a non-increasing sequence of regularization parameters $(\lambda_i)_{i \in \llbracket 1,p \rrbracket}$ such that $\lambda_1 \geq \dots \geq \lambda_p \geq 0$ (with at least one strict inequality), we consider a sorted penalty of the form
    \begin{equation}
        \varPsi : \rvx \mapsto \sum_{i=1}^p \psi(|x_{(i)}|; \lambda_i),
    \end{equation}
    where the scalar penalty function $\psi : \sR_+ \times \sR_+ \to \sR_+$ satisfies the following properties, for all $\lambda >0$,
\begin{enumerate}
    \item $\psi(\cdot ; \lambda)$ is continuous, non-decreasing and concave on $\sR_+$, and $\psi(0;\lambda) = 0$
    \item $\psi(\cdot ; \lambda)$ is continuously differentiable on $\sR_{+*}$ and $\psi'(0;\lambda) := \lim_{z \to 0} \psi'(z;\lambda) \in \sR_+ \cup \{ + \infty \}$.
\end{enumerate}
    Moreover, for all $z\geq 0$, $\psi'(z; \cdot)$ is non-decreasing where the derivative, as always in this paper, is with respect to the first argument $z$.
\end{hypothesis}

This includes in particular sorted versions of MCP, SCAD, log-sum or $\ell_q$ penalties.

A popular algorithm to solve \Cref{pb:composite} is proximal gradient descent (also referred to as forward-backward splitting \cite{combettes2005signal}), which iterates
\begin{equation}
   \label{equation:pgd_iterate}
    \rvx^{k+1} = \prox_{\eta \varPsi} \left(\rvx^k - \eta \nabla g(\rvx^k) \right) \, ,
\end{equation}
where $\eta>0$ is the stepsize; refinements such as FISTA~\cite{beck2009fast} improve the convergence speed.
Many other optimization algorithms (e.g., ADMM, primal-dual methods) rely on the proximal operator of $\varPsi$, which highlights the importance of being able to compute it efficiently.
The convergence of proximal methods in nonconvex settings is not the main focus of this paper but a vast amount of litterature is dedicated to it, both in the weakly convex regime \cite{bayram2015convergence} and the non-convex regime \cite{attouch2013convergence}.
The rest of the paper is thus devoted to the computation of $\prox_{\eta \varPsi}$, for $\varPsi$ satisfying \Cref{hyp:gen}.

\subsection{Basic Properties}

First, we present general properties of proximal operators for sorted penalties.
We recall here, in a more general form, results already known in the case of SLOPE (\Cref{def:slope}). The proofs are in Supplementary Material.
Proximal operators were first studied by \cite{moreau1965proximite} in the case of proper, lower semi-continuous (l.s.c.) and convex functions.
Because most functions under scrutiny in this work are nonconvex, we will introduce here a more general definition (referred to as \emph{proximal mapping}) following the framework of \cite[Chapter 6]{beck2017first}.
\begin{definition}[Proximal mapping]
    \label{def:prox}
The proximal mapping of a function $\varPsi : \sR^p \to (-\infty, + \infty]$ is, for any $\rvy \in \sR^p$:
\begin{equation*}
    \prox_\varPsi(\rvy) = \argmin_{\rvx \in \sR^{p}} \left\{ \varPsi(\rvx) + \frac{1}{2} \Vert \rvx - \rvy \Vert^2 \right\} \subset \sR^p \, .
\end{equation*}
\end{definition}

\begin{proposition}[{Non-emptiness of the prox under closedness and coerciveness, \cite[Thm. 6.4]{beck2017first}}]
    \label{prop:nonempty-prox}
    If $\varPsi$ is proper, l.s.c. and $\varPsi + \frac{1}{2}||\cdot -\rvx||^2$ is coercive for all $\rvx \in \sR^p$, then $\prox_\varPsi(\rvx)$ is non-empty for all $\rvx \in \sR^p$.
\end{proposition}

In our setup, since $\varPsi$ is bounded below by $0$ (\Cref{hyp:gen}) and thus $\eta \varPsi + \frac{1}{2}||\cdot -\rvx||^2$ is coercive, $\prox_{\eta \varPsi}$ is non-empty for any stepsize $\eta >0$.
As the next proposition shows, to compute the proximal operator of a sorted penalty, it is enough to be able to compute it on non-negative and non-increasing vectors.

\begin{proposition}
    \label{prop:comput_prox}
    For any $\rvy \in \mathbb{R}^p$, we can recover $ \prox_{\eta \varPsi}(\rvy)$ from  $\prox_{\eta \varPsi}(|\rvy|_\downarrow)$ by:
    \begin{equation*}
        \prox_{\eta \varPsi}(\rvy) = \sign(\rvy) \odot \mathbf P_{|\rvy|}^\top \prox_{\eta \varPsi}(|\rvy|_\downarrow) \, ,
   \end{equation*}
   where $\mathbf P_{|\rvy|}$ denotes any permutation matrix that sorts $|\rvy|$ in non-increasing order:  $\mathbf P_{|\rvy|}(|\rvy|) = |\rvy|_\downarrow$.
\end{proposition}

Denoting $\gK_p^+$ the cone of positive vectors $\rvy \in \mathbb{R}^p$ satisfying $y_1 \geq \dots \geq y_p \geq 0$, we get from \Cref{prop:comput_prox} that the difficulty in computing the prox of a sorted penalty at any point $\rvy \in \sR^p$ is the same as computing this prox for sorted positive vectors $\rvy \in \gK_p^+$ only.
Moreover, it turns out that the latter can be replaced by an equivalent non-negative isotonic\footnote{as in isotonic regression, to be understood as monotonic} constrained problem, as stated by \Cref{prop:prox_cone}.
\begin{proposition}[Proximal operator of non-negative vectors]
    \label{prop:prox_cone}
    Let $\rvy \in \gK_p^+$, i.e.  $y_1 \geq \dots \geq y_p \geq 0$.
    Then:

     \begin{equation}
        \prox_{\eta \varPsi}(\rvy) = \argmin_{\rvx \in \gK_p^+}  P_{\eta \psi}(\rvx; \rvy) \, ,
       \label{eq:prox_way2}
    \end{equation}
    where
    \begin{equation}\label{eq:unsort_prox_obj}
      P_{\eta \psi}(\rvx; \rvy) :=  \sum_{i=1}^p \psi (x_i, \lambda_i) + \frac{1}{2\eta} \Vert \rvy - \rvx \Vert^2 \ .
    \end{equation}
\end{proposition}
Note that this minimization problem depends solely on the \emph{unsorted penalty} constrained to a \emph{convex} set.
To sum up, if one is able to solve the problem given in \Cref{eq:prox_way2}, then from \Cref{prop:comput_prox}, one can recover the proximal operator of $\eta \varPsi $ at any $\rvy \in \sR^p$.

\subsection{The PAV Algorithm for Isotonic Convex Problems}

The PAV algorithm has been introduced by the authors in \cite{best1990active} to solve the isotonic regression problem, i.e. the orthogonal projection onto the isotonic cone $\gK_p^+$.
The authors in \cite{best2000minimizing} have extended it to the minimization of separable convex functions over $\gK_p^+$.

We give PAV pseudo code in \appref{appendix:pava} and briefly explain below how it operates.
In simple terms, the PAV algorithm starts from the unconstrained optimal point and partitions the solution into blocks of consecutive indices on which the value of the solution is constant (these blocks are singletons at initialization).
It then iterates through the blocks, merging them using \emph{a pooling operation} when the monotonicity constraint is violated:
\begin{itemize}
    \item \textit{Forward step:} PAV merges the current working block with its successor if monotonicity is violated.
    \item \textit{Backward step:} Once this forward pooling is done, the current working block is merged with its predecessors as long as monotonicity is violated.
\end{itemize}

Finally, the additional non-negative condition is dealt with by taking the positive part of the solution at the end of the algorithm (i.e. projecting on the non-negative cone).
PAV is of linear complexity when the pooling operation is done in $\mathcal O(1)$.
The correctness proof of the PAV algorithm~\cite[Theorem 2.5]{best2000minimizing} consists in showing that the KKT conditions are satisfied when the algorithm ends.
The PAV algorithm works for \emph{any separable convex functions} as long as one knows the pooling operation to be done when updating the values.

\section{Proximal Operator of Sorted \\Weakly Convex Penalties}
\label{section:weakly-convex}

In this section, we describe how the PAV algorithm can be used to efficiently compute the proximal operator of many sorted nonconvex  penalties, such as sorted SCAD, sorted MCP or sorted log-sum, for a stepsize small enough, by leveraging weak convexity.

\begin{hypothesis}[Weak convexity]
    \label{assumption:scalar_penalty_weak}
    In addition to \Cref{hyp:gen}, assume that for all $i \in \intervalle{1}{p}$, $\psi(\cdot; \lambda_i)$ is \emph{$\mu$-weakly convex}. In other words, there exists some $\mu > 0$ such that, for any $\alpha \geq \mu$ and any $i \in \intervalle{1}{p}$, $\psi(\cdot;\lambda_i) + \frac{\alpha}{2} \norm{\cdot}^2$ is convex.
\end{hypothesis}

The next proposition shows that computing the proximal operator under \Cref{assumption:scalar_penalty_weak} comes down to solving a convex problem.

\begin{proposition}
    \label{proposition:weak_convex_exist_unique_prox}
    If the sorted penalty $\varPsi$ satisfies \Cref{assumption:scalar_penalty_weak} for some $\mu>0$, then, for $0 < \eta < \frac{1}{\mu}$, the unsorted objective $P_{\eta \psi}$ in \Cref{eq:unsort_prox_obj} is convex and %
        the proximal mapping $\prox_{\eta \varPsi}$ is a singleton on~$\sR^p$.
\end{proposition}

\begin{proof}
    The proximal mapping is non-empty for any $\rvx \in \sR^p, \eta > 0$.
     Then $P_{\eta \psi}$ is strongly convex under  \Cref{assumption:scalar_penalty_weak} and $\eta < \frac{1}{\mu}$.
    From \Cref{prop:comput_prox,prop:prox_cone}, as the set of minimizers of a strongly convex function, $\prox_{\eta \varPsi}(\rvy)$ is a singleton for all $\rvy \in \sR^p$.
\end{proof}

\begin{remark}
    \Cref{proposition:weak_convex_exist_unique_prox} requires that the stepsize $\eta$ is smaller than a constant depending on the weak-convexity parameter of the penalty. Although it may seem restrictive, note that it is consistent with the condition for convergence of the PGD algorithm (\Cref{equation:pgd_iterate}) that imposes $\eta < \frac 1 L$ with a $L$-smooth datafit $f$.
\end{remark}

\subsection{Prox Computation with the PAV Algorithm}

With $\eta < 1/\mu$ the prox problem becomes a convex isotonic problem so the PAV algorithm can be used.

The pooling operation which updates the value on a block $B$, denoted $\chi(B)$, is defined as:
\begin{equation}
    \label{eq:chi_b}
    \chi(B) := \argmin_{z \in \sR_+} \sum_{i \in B} \frac{1}{2\eta} (z-y_i)^2 + \psi(z;\lambda_i),
\end{equation}
where we recall that here $\rvy \in \gK_p^+$.
\Cref{proposition:chi_prox_ope} shows that, for sorted weakly convex penalties, the pooling operation given by \Cref{eq:chi_b} reduces to a \emph{scalar proximal operation}.

\begin{proposition}[Writing $\chi$ as a proximal operation]
    \label{proposition:chi_prox_ope}
    Under \Cref{assumption:scalar_penalty_weak},
    for a block of consecutive indices $B$, the pooling operation $\chi(B)$ is the proximal point of an averaged scalar penalty evaluated at an averaged data point.
    \begin{equation}
        \chi(B) = \prox_{\frac{\eta}{\card{B}} \sum_{i\in B} \psi(\cdot; \lambda_i)} (\bar y_B) \ ,
    \end{equation}
    where we use the notation $\bar y_B$ as the average value of the vector $\rvy$ on the block $B$: $\bar y_B = \frac{1}{\vert B \vert}\sum_{i \in B} y_i$.
    Besides, if $\psi(\cdot; \lambda_i) = \lambda_i \psi_0(\cdot)$, then we have %
    \begin{equation}
        \chi(B) = \prox_{\bar \lambda_B \eta \psi_0}(\bar y_B) \ .
    \end{equation}
\end{proposition}

\begin{proof}
    Let $\rvy \in \gK_p^+$ be the point for which we want to compute $\prox_{\eta\Psi}(\rvy)$.
    Let $B = \intervalle{q}{r}$ be a block of consecutive indices.
    Then, $\chi(B)$ is the unique zero of the function $f_{B}$ defined as
    \begin{align*}
        f_{B} (z) &= \card{B} \left[ \frac{1}{\eta} \left( z - \bar y_B \right) + \frac{\sum_{i \in B} \psi'(z;\lambda_i)}{\card{B}}\right] \\
        & = \card{B} \frac{\partial}{\partial z}  \left[ \frac{1}{2\eta} \left( z - \bar y_B \right)^2 + \frac{\sum_{i \in B} \psi(z;\lambda_i)}{\card{B}}\right] \ .
    \end{align*}
    So, $\chi(B)$ is the unique minimizer of the strongly convex function $z \mapsto \frac{1}{2\eta} \left( z - \bar y_B \right)^2 + \frac{\sum_{i \in B} \psi(z;\lambda_i)}{\card{B}}$, which gives the expected result.
\end{proof}

\begin{remark}
    For a standard non-sorted penalty, one would apply the proximal operator component-wise.
    \Cref{proposition:chi_prox_ope} highlights that the use of sorted penalties enforces the proximal operation to be done block-wise, which enforces coefficient grouping, as intended.
\end{remark}

\begin{example}[Sorted MCP]\label{ex:smcp}
    For any $\gamma >0$, $\lambda >0$, the MCP penalty is defined on $\sR_+ \times \sR_+$ as:
    \begin{equation*}
        \psi : (z;\lambda) \mapsto
        \begin{cases}
            \lambda z - \frac{z^2}{2\gamma} & \text{if } z \leq \gamma \lambda \ ,\\
            \frac{\lambda^2\gamma}{2} & \text{otherwise} \ .
        \end{cases}
    \end{equation*}
    It is $\frac{1}{\gamma}$-weakly convex, so is $\varPsi$.
    From~\Cref{proposition:chi_prox_ope}, the solution $\chi(B)$ of \Cref{eq:chi_b} is the zero of the continuous, strictly increasing affine function $f_B$:
    \begin{equation*}
        f_B: z \mapsto  \frac{1}{\eta} \card{B} \left( z - \bar y_B \right) + \sum_{i\in B} \left(\lambda_i - \tfrac{z}{\gamma}\right)_+ \ .
    \end{equation*}
    By denoting for every $i \in B=\intervalle{q}{r}$, $v_i := \left[ \card{B}\gamma - \eta (i-q+1)\right] \lambda_i + \eta \sum_{j=q}^i \lambda_j$, we have an explicit expression for $\chi(B)$ that only depends on the position of $\bar y_B$ in the sequence $(v_i)_{i \in B}$:
    \begin{equation*}
        \chi(B) = \frac{\bar y_B  - \frac{\eta}{\card{B}} \sum_{j=q}^i \lambda_i}{1 - \frac{i-q+1}{\card{B}} \frac{\eta}{\gamma}} \, ,
    \end{equation*}
    with $i$ such that $\card{B}  v_{i+1} \leq \bar y_B <  \card{B}  v_i $ (we denote $v_{q-1}= + \infty$, $v_{r+1} = 0$ and $\sum_{j=q}^{q-1} = 0$).
\end{example}

\begin{example}[Sorted Log-Sum]\label{ex:sort_log_sum}
    For any $\epsilon >0$, $\lambda >0$, the log-sum penalty is defined on $\sR_+ \times \sR_+$ as $$\psi : (z;\lambda) \mapsto \lambda \log \left(1 +\frac{z}{\epsilon} \right).$$
    It is $\frac{\lambda}{\epsilon^2}$-weakly convex.
    From~\Cref{proposition:chi_prox_ope} and
     the known formula of the scalar proximal operator of log-sum \cite{prater2022proximity}, we can derive the following formula:
    \begin{equation*}
        \chi(B) =
        \begin{cases}
            0 & \text{if } \bar y_B \leq \frac{\eta \bar \lambda_B}{\epsilon} \ , \\
            \frac 1 2 (\bar y_B -\epsilon) + \sqrt{\frac 1 4 (\bar y_B + \epsilon)^2 - \eta \bar \lambda_B} & \text{otherwise} \ .
        \end{cases}
    \end{equation*}
    hence, in the weakly convex case the prox can be computed as easily as for the sorted $\ell_1$ norm.
\end{example}

\subsection{Links with the Majorization-Minimization Approach}\label{sub:mm}
The related work \cite{feng2019sorted} proposed a majorization-minimization (MM) framework in order to use  sorted nonconvex penalties; which, contrary to ours, avoids computing the prox directly.
We consider here the example of sorted MCP, denoted $\varPsi_{\mathrm{SMCP}}$, for which a detailed algorithm is given in Algorithms 2 to~4 of \cite{feng2019sorted}
It relies on the decomposition of the penalty as a difference of two convex functions. While there is no single such decomposition, given the piecewise quadratic nature of MCP, a natural one is given by
\begin{multline}
	\varPsi_{\mathrm{SMCP}}(\rvx, \gamma, (\lambda_i)_i) = \\ \underbrace{\varPsi_{\mathrm{SMCP}}(\rvx, \gamma, (\lambda_i)_i)+  \frac{1}{2\gamma}\|\rvx\|^2}_{\varPsi^+(\cdot, \gamma, (\lambda_i)_i) \text{ (convex sorted penalty)}} -  \underbrace{\frac{1}{2\gamma}\|\rvx\|^2}_{\varPsi^-}.
\end{multline}
This allows the authors of \cite{feng2019sorted} to easily deploy a MM  approach through the majorization of the concave term ($-\varPsi^-$)  by its tangent at the current point. More precisely, the MM iterations are given by
\begin{equation}
    \label{eqn:mm_iter}
	\rvx^{k+1} = \mathrm{arg} \min_{\rvx} \; g(\rvx)  - \rvx^\top \nabla \varPsi^-(\rvx^k) + \varPsi^+(\rvx,\gamma,(\lambda_i)_i)
\end{equation}
where $g$ denotes the datafit, assumed to be differentiable. The minimization of this surrogate function is then tackled with a proximal gradient method.
This requires the computation of  $\prox_{\varPsi^+}$ which can be done using a PAV algorithm.
Differently, our approach leverages the weak convexity of the sorted MCP penalty to use proximal algorithms (such as ISTA, FISTA) directly on the original problem.
We experimentally highlight the difference between these two methods on two problems: penalized least square regression and penalized logistic regression.
The experimental setup is detailed in \appref{appendix:mm}.

\begin{figure}
    \centering
    \includegraphics[width=0.49\linewidth]{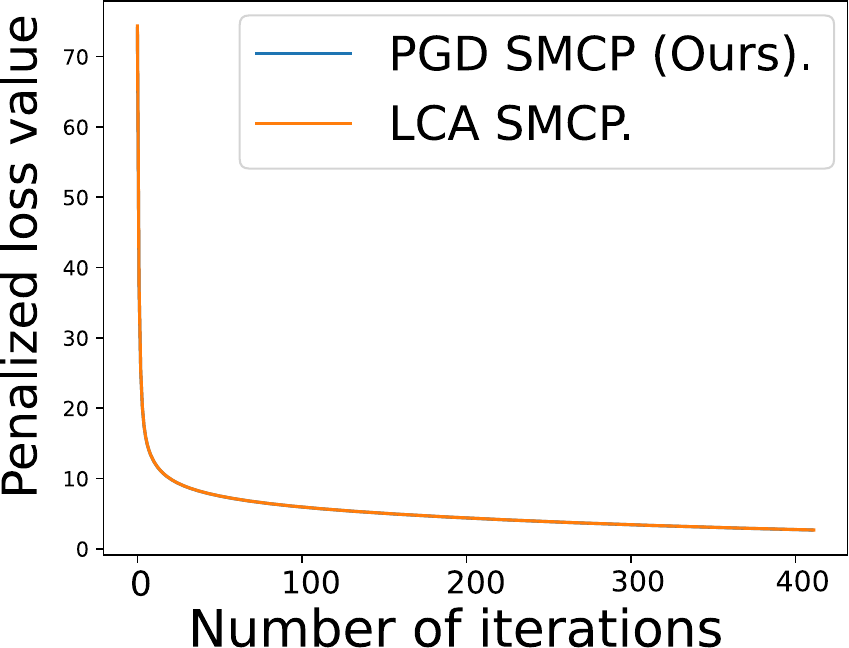}
    \includegraphics[width=0.49\linewidth]{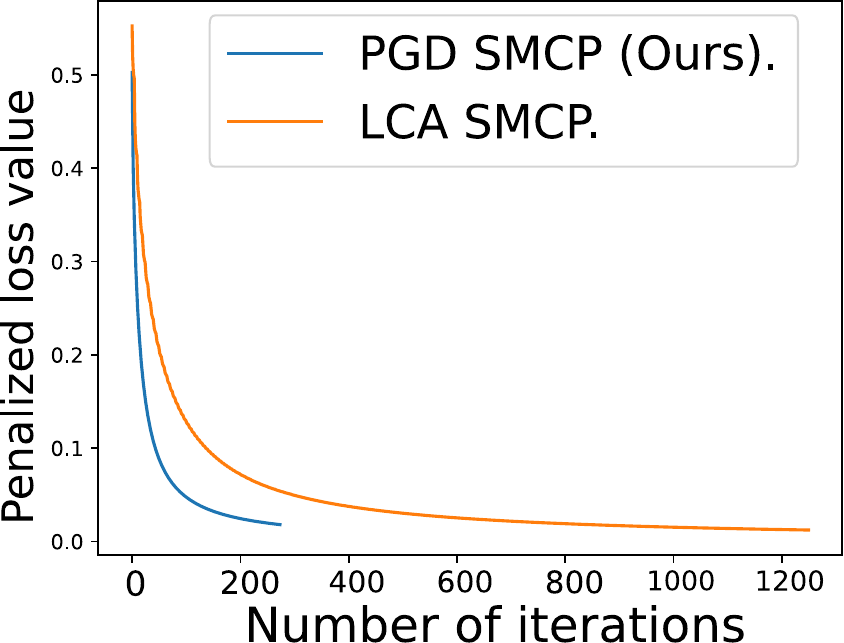}
    \caption{Comparison of the MM approach of \cite{feng2019sorted}  (LCA SMCP) with the proximal gradient method using our proposed computation of the prox on least squares (left) and logistic regression (right) problems.
    }
    \label{fig:mm}
\end{figure}

\Cref{fig:mm} displays the evolution of the penalized loss value which we aim to minimize. We can make two observations:
\begin{enumerate}
    \item For the least square regression problem, the MM approach from \cite{feng2019sorted} (denoted as LCA SMCP for Linear Convex Approximation Sorted MCP) behaves the same as our method.
    The convex approximation relies on a quadratic rectification added to the datafit which does not change the behaviour of the loss for the least square case.
    \item For the logistic regression problem, the MM approach from \cite{feng2019sorted} has slower convergence rate compared to our direct approach.
\end{enumerate}

\section{Proximal Operator of Sorted \\ Nonconvex Penalties}
\label{section:nonconvex}

Contrary to the examples of \Cref{section:weakly-convex}, in the most general case, the prox objective function may not be strongly convex, and we need a more refined analysis to devise prox computation strategies.
This section addresses a broader class of sorted penalties  defined by the following hypothesis.

\begin{hypothesis}
    \label{hyp:nonconvex}
    In addition to \Cref{hyp:gen}, assume that the scalar penalty function is such that $\psi(\cdot ; \lambda) = \lambda \psi_0(\cdot)$ and twice differentiable on $\sR_{+*}$, except potentially at a finite set of points, denoted $Z_\psi$.
    The second derivative $\psi_0''$ is increasing and such that $\lim_{z\to + \infty} \psi_0''(z) = 0$.
\end{hypothesis}

This class of penalties contains some of the weakly convex penalties studied in the previous section such as log-sum but does not restrict them to the weakly convex regime. It also includes many other nonconvex penalties~\cite{antoniadis2011penalized} such that, for instance,
the $\ell_q$ penalty for $q \in ]0,1[$ defined as $\psi_0(x) = x^q$, which is not weakly convex.
In the unsorted case, several experimental studies \cite{li2018nonconvex,yuan2019analysis} show that $\ell_q$ outperforms $\ell_1$ or SCAD penalties for specific inverse problems: its recurrent appearance in the  literature makes it worthy of interest.

\subsection{The Scalar Case}

For $y \geq 0$ and $\lambda > 0$, we consider throughout this section the scalar proximal objective function of $\lambda \psi_0$,  denoted $F$
\begin{equation}
  F : z \in \sR_+ \mapsto  \frac{1}{2} (z-y)^2 +  \lambda \psi_0(z),
\end{equation}
such that $\prox_{\lambda \psi_0}(y) = \argmin_{z \in \sR_+} F(z)$.\footnote{Note that, under \Cref{hyp:nonconvex}, $\eta \psi(\cdot;\lambda) = \eta \lambda \psi_0$. In the remaining of this section, we consider wlog $\eta=1$. One can simply replace every occurrence of $\lambda$ by $\eta \lambda$ to recover the dependence on $\eta$.}
From \Cref{hyp:nonconvex}, we get that $F$ is continuous, non-negative and coercive (at $z \to \infty$) which guarantees the existence of at least one minimizer in $\sR_+$.
The next proposition details the variations of the function $F$.

\begin{proposition}[Variations of scalar nonconvex prox objective]
    \label{prop:variation_scalar}
    For fixed $\lambda >0$ and $y>0$, the prox objective $F$ has the following properties:
\begin{enumerate}[label=(\roman*)]
    \item \label{prop:variations_item_1} \textbf{Concavity-convexity transition:} There exists $m(\lambda) \geq 0$ such that the function $F$ is strictly concave on $[0,m(\lambda)]$ and convex on $[m(\lambda), +\infty)$.
    \item \label{prop:variations_item_2} \textbf{Local minimizers transition:} $F$ admits at most two local minimizers on $\sR_+$.
    There exists $\tau(\lambda)$ such that
    \begin{itemize}
    \item if $y < \tau(\lambda)$, $0$ is the unique local minimizer (hence global),
    \item if $ \tau(\lambda)\leq y \leq   \lambda \psi'_0(0)$, there exist two local minimizers: $0$ and another one, denoted $\rho^+(y; \lambda)$,
    \item if $y \geq   \lambda \psi_0'(0)$, there is a unique minimizer, still denoted  $\rho^+(y; \lambda)$. %
    \end{itemize}
    Moreover, the threshold $\tau(\lambda)$ writes
    \begin{equation}
        \tau(\lambda) = m(\lambda) +  \lambda \psi_0'(m(\lambda)),
    \end{equation}
    and the nonzero minimizer admits the  lower bound
    \begin{equation}
        \rho^+(y; \lambda) \geq m(\lambda) > 0
    \end{equation}
    Finally,  $\rho^+$ is continuous with the following variations: it is non-decreasing with $y$ and non-increasing with~$\lambda$.
    \item \label{prop:variations_item_3} \textbf{Global minimizers transition:} There exists $T(\lambda) \in [ \tau(\lambda),  \lambda \psi_0'(0)]$ such that
    \begin{itemize}
        \item if $y<T(\lambda)$, $0$ is the global minimizer,
        \item if $y=T(\lambda)$, $0$ and $\rho^+(y;\lambda)$ are global minimizers.
        \item if $y>T(\lambda)$, $\rho^+(y;\lambda)$ is the global minimizer,
    \end{itemize}
\end{enumerate}
These results are illustrated on \Cref{fig:shapes_F}.
\end{proposition}

This proposition, valid for any penalty satisfying \Cref{hyp:nonconvex}, generalizes some peculiar results that can be found in the literature, namely for log-sum or $\ell_q$, that we recall below.
The proof is in Appendix~\ref{proof:min_local_scalar}.
\begin{example}[Log-sum]
        For the log-sum penalty, defined in \Cref{ex:sort_log_sum}, \cite[Lemma 2, Proposition 2]{prater2022proximity} gives the existence of the two thresholds $\tau(\lambda)$ and $T(\lambda)$:
        \begin{enumerate}
            \item $\tau(\lambda) = \min(2 \sqrt{\lambda} - \epsilon,\lambda/\epsilon)$ decides whether there is one minimizer which is $0$ or two local minimizers, $0$ and $\rho^+$ on $\sR_+$.
            \item $T(\lambda) \geq \tau(\lambda)$, defined implicitely but such that $T(\lambda) \leq {\lambda }/{\epsilon}$, decides whether $0$ or $\rho^+$ is the global minimizer.
        \end{enumerate}
        The local minimizer $\rho^+$ has an explicit formula
        \begin{equation}
            \rho^+(y, \lambda) = \frac{1}{2}(y-\epsilon) + \sqrt{\frac{1}{4} (y + \epsilon)^2 - \lambda } \ .
        \end{equation}
       Moreover, \cite[Lemma 3]{prater2022proximity} provides an explicit lower bound on $\rho^+$, which is non-decreasing with $\lambda$: for $y \geq \tau(\lambda)$, then $\rho^+(y;\lambda) \geq \max(\sqrt{\lambda } - \epsilon,0) = m(\lambda)$.
\end{example}

\begin{example}[$\ell_q$]
    For the $\ell_q$ penalty, defined as $\psi(x;\lambda) = \lambda x^q$ for $q \in (0,1)$,  the two thresholds $\tau(\lambda)$ and $T(\lambda)$ are given by
    \begin{align*}
          \tau(\lambda) &:= \left[ \frac{2-q}{1-q} \right] \left(  \lambda  q(1-q) \right)^{\frac{1}{2-q}} \\
          T(\lambda) & := \left[\frac{1}{2} \frac{2-q}{1-q} \right] \left( 2  \lambda  (1-q) \right)^{\frac{1}{2-q}}
    \end{align*}
    Proofs can be found in \cite{xu2012l_,chen2016computing}.
    \Cref{fig:shapes_F} illustrates the local minimizers transition cases.
    Moreover, among the $\ell_q$ penalties, $\ell_{1/2}$ and $\ell_{2/3}$ stand out as two penalties whose (scalar) proximal operator can be computed in closed form (i.e., for which we have an expression of $\rho^+$) \cite{xu2012l_,chen2016computing}.
\end{example}

\begin{figure}[t]
    \centering
    \newcommand{\mysize}{0.24}
    \begin{subfigure}[t]{\mysize\linewidth}
        \includegraphics[width=\linewidth]{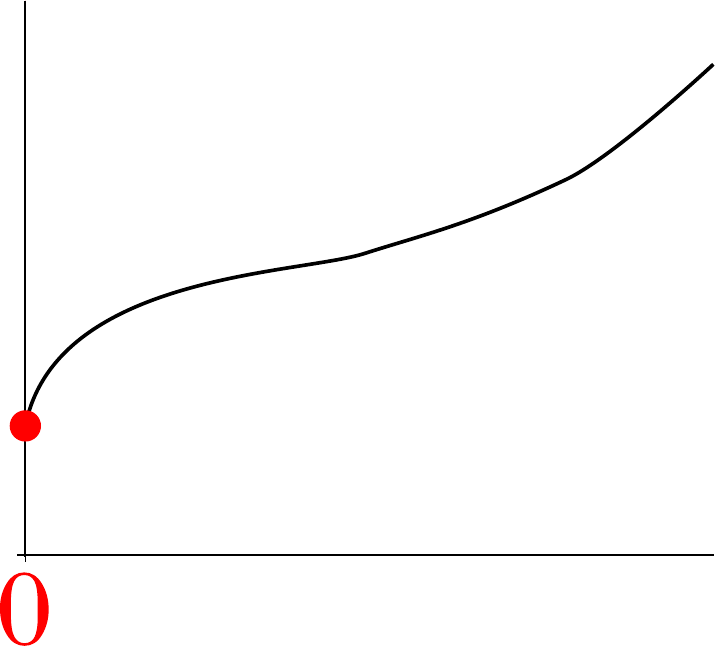}
    \subcaption{\small{$ y <  \tau(\lambda)$}}
    \label{fig:no_chi_0_min}
    \end{subfigure}
    \begin{subfigure}[t]{\mysize\linewidth}
        \includegraphics[width=\linewidth]{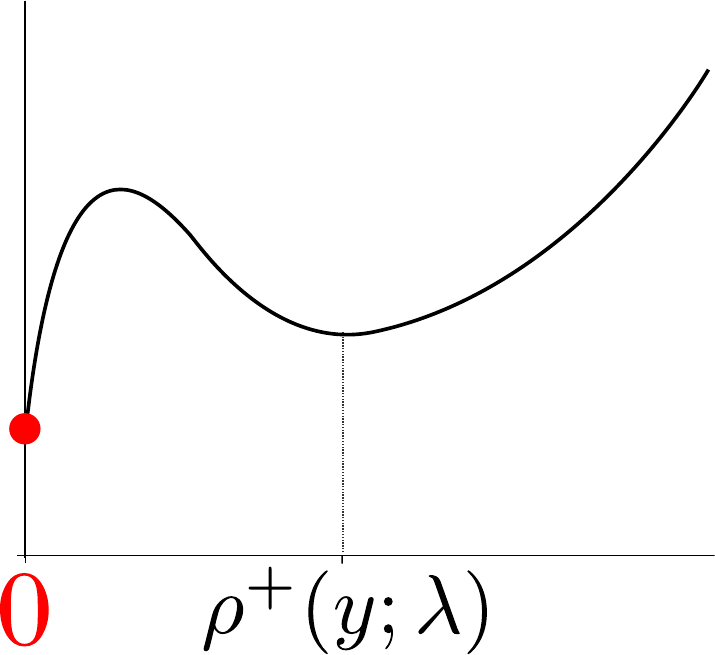}
    \subcaption{\small{$\tau(\lambda) \leq y \leq T(\lambda)$}}
    \label{fig:chi_0_min}
    \end{subfigure}
    \begin{subfigure}[t]{\mysize\linewidth}
        \includegraphics[width=\linewidth]{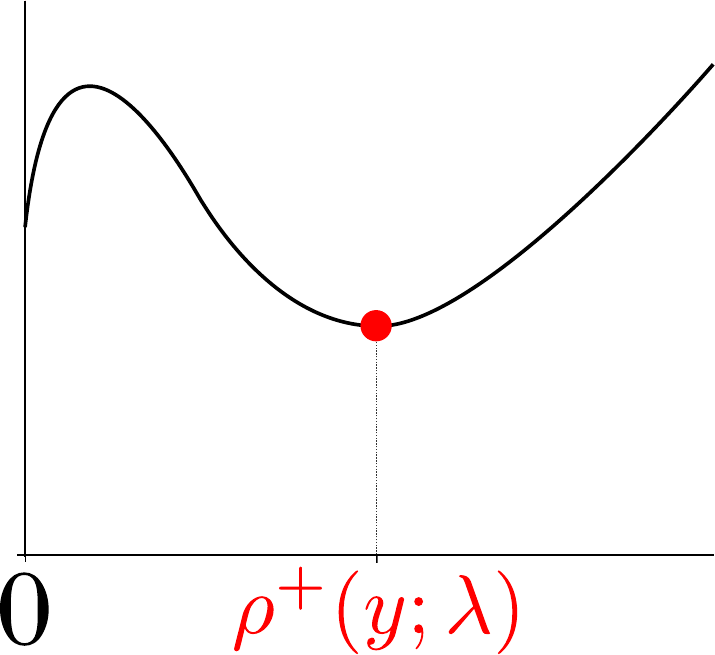}
    \subcaption{\small{$ T(\lambda) < y < \lambda \psi'_0(0)$}}
    \label{fig:no_chi_min}
    \end{subfigure}
    \begin{subfigure}[t]{\mysize\linewidth}
        \includegraphics[width=\linewidth]{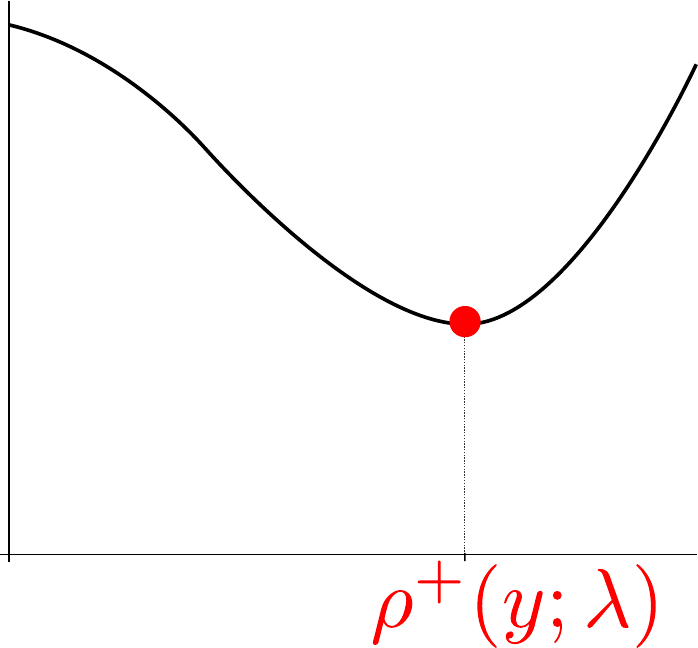}
    \subcaption{\small{$  \lambda \psi'_0(0) < y$}}
    \end{subfigure}

    \caption{The $4$ (or $3$ when $\psi_0'(0) = +\infty$, e.g., $\ell_q$)
     possible variation profiles of $F$, with its local minimizers, depending on the value of $y$. The red dot denotes the global minimizer.
    }
    \label{fig:shapes_F}
\end{figure}

\subsection{Finding Local Minimizers of the Proximal Problem}

For the considered class of penalties, \Cref{prop:prox_cone} holds and so our proximal problem reduces to a \emph{nonconvex} isotonic regression problem.
For $\rvy \in \gK_p^+$ the prox problem is
\begin{equation}
    \label{problem:prox_nonconvex}
    \underset{\rvx \in \gK_p^+}{\mathrm{arg} \, \min} \; \frac{1}{2} \norm{\rvx-\rvy}^2 + \sum_{i=1}^p  \lambda_i \psi_0(x_i) \ .  \tag{\text{$P_p$}}
\end{equation}

In this section, we describe the set of \emph{local minimizers} of Problem $(P_p)$ and show that the standard PAV algorithm converges to one of them.

For a block $B = \intervalle{q}{r}$, we define as $F_B$, or alternatively $F_{q:r}$ the \emph{scalar} prox objective function on $\sR_+$:
\begin{align}\label{eq:scalar_func_B}
    F_B = F_{q:r} : z \mapsto & \sum_{i \in B} \frac{1}{2}(y_i - z)^2 +  \lambda_i \psi_0(z)  \\
    \propto &  \; \frac12 (\bar y_B - z)^2 +  \bar \lambda_B \psi_0(z) + C \label{eq:scalar_func_B-2}
\end{align}
with $C \in \R$ a constant independent of $z$, and $\bar y_B$ (resp. $\bar \lambda_B$)  the average value of the vector $\rvy$ (resp. of the $\lambda_i$'s) on the block $B$, i.e. $\bar y_B = \frac{1}{|B|}\sum_{i \in B} y_i$.
Given a block of indices~$B$,~\eqref{eq:scalar_func_B-2} shows that minimizing the scalar function $F_B$ comes down to solving $\prox_{ \bar \lambda_B \psi_0}$ in 1D.

\begin{theorem}[Characterization of local minimizers]
    \label{theorem:ell_q_local}
    Let $\Psi$ be a sorted penalty satisfying \Cref{hyp:nonconvex} with $(\lambda_i)_{i \in \intervalle{1}{p}}$ the sequence of regularization parameters. The threshold $\tau$ is defined as in  \Cref{prop:variation_scalar}~\ref{prop:variations_item_2}.
    For a sorted vector $\rvy \in \gK^+_p$, we consider the prox problem \eqref{problem:prox_nonconvex} and denote $\gL_p$ the set of its local minimizers.
    Then, $\rvu \in \gL_p$ if and only if, on each maximal\footnote{meaning that if $B$ is extended, $\rvu$ is no longer constant on it} block $B = \intervalle{q}{r}$ where $\rvu$ is constant equal to $\tilde{u}$, the two following statements hold:
    \begin{enumerate}[label=(\roman*)]
        \item \label{theorem:ell_q_local_point_1}$\tilde{u} \in \{\chi(B),0\}$, where
        \begin{equation}
            \label{equation:ell_q_chi}
            \chi(B) :=
            \begin{cases}
                \rho^+(\bar y_B,  \bar \lambda_B) & \text{if } \bar y_B \geq \tau(  \bar \lambda_B)  \ ,\\
                0 & \text{otherwise} \ .
            \end{cases}
        \end{equation}
        \item  \label{theorem:ell_q_local_point_2} when $\tilde{u} = \chi(B)$, then for all $j \in \intervalle{q}{r-1}$ we have,
        \begin{align}
            \text{either } &\chi(\intervalle{q}{j}) \leq \chi(B) \leq \chi(\intervalle{j+1}{r}),\label{theorem:ell_q_local-1} \\
            \text{either } &\chi(\intervalle{q}{j}) \geq  \chi(\intervalle{j+1}{r}) \geq \chi(B).\label{theorem:ell_q_local-2}
        \end{align}
        Moreover if $\chi(B)>0$, $F_{q:j}$ is increasing at $\chi(B)$ and $F_{j+1:r}$ is decreasing at $\chi(B)$ while if $\chi(B)=0$ both $F_{q:j}$ and $F_{j+1:r}$ are increasing at $\chi(B)$.
    \end{enumerate}
\end{theorem}

\begin{proof}
    The detailed proof is in \suppref{appendix:proofs}; it relies on a careful study of the different possible variation profiles for $F_B$. Informally, the proof relies on two main ideas:
    \begin{enumerate}[label=(\roman*)]
        \item \Cref{theorem:ell_q_local}~\ref{theorem:ell_q_local_point_1} is obtained from the fact that the value on each block $B$ of a local minimizer must itself be a local minimizer of the corresponding scalar subproblem on the block (i.e. minimizing $F_B$).
        \item \Cref{theorem:ell_q_local}~\ref{theorem:ell_q_local_point_2} is obtained from the fact that any feasible infinitesimal move around a local minimizer (i.e., splitting a block) should not decrease the objective value. This is illustrated on \Cref{fig:loc_min}. \qedhere
    \end{enumerate}
\end{proof}

\begin{figure}
    \centering
    \includegraphics[width=\linewidth]{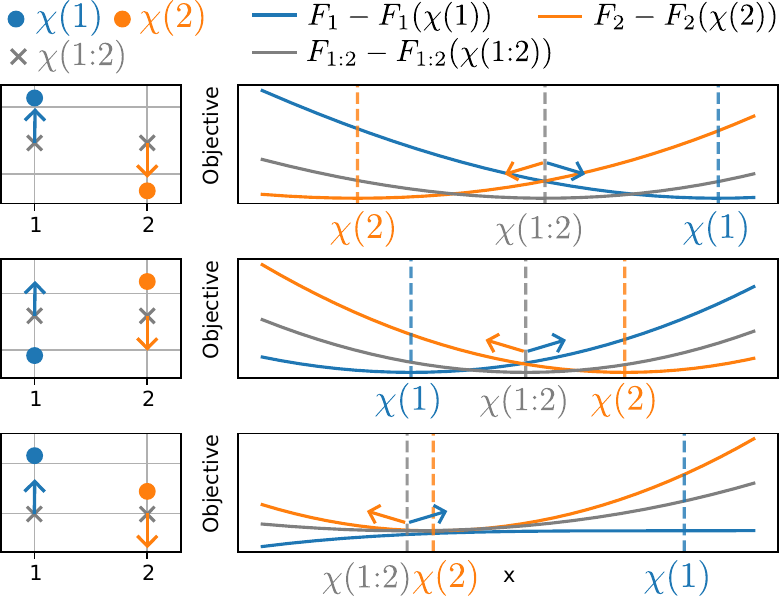}
    \caption{Illustration of \Cref{theorem:ell_q_local}~\ref{theorem:ell_q_local_point_2} with $B = \{1 , 2\}$. Left: Values of $\chi$. Right: Normalized objective functions (with $F_{1:2} = F_1 + F_2$). \textbf{Top row:} Correctly ordered blocks, merged value lies between; \emph{not} a local minimum since splitting decreases $F_1$ and $F_2$. \textbf{Middle row:} Incorrectly ordered blocks, merged value lies between; a local minimizer since splitting the block increases $F_1$ and $F_2$ \eqref{theorem:ell_q_local-1}. \textbf{Bottom row:} Correctly ordered  blocks, merged value is lower; a local minimizer since splitting the block increases $F_1$ and $F_2$ \eqref{theorem:ell_q_local-2}. Note that, from \Cref{lemma:merge,lem:bloc_above_max_subbloc} in Supplementary Material, these are the only three possible configurations to analyze.}
    \label{fig:loc_min}
\end{figure}
The second point of \Cref{theorem:ell_q_local} highlights a key difference with the weakly convex regime and the use of the PAV algorithm.
In the weakly convex regime, PAV algorithm merges blocks that violate the isotonic constraint. Yet, in the non weakly convex regime, \Cref{theorem:ell_q_local-2} points out a more pathological behavior: merging correctly ordered blocks, which PAV does not do, might actually lead to better solutions (as displayed in the bottom row of \Cref{fig:loc_min}).
Nevertheless, using \Cref{theorem:ell_q_local}, we are able to show that, while being non exhaustive, the PAV algorithm still reaches a local minimizer of the proximal problem of sorted penalties satisfying \Cref{hyp:nonconvex}.

\begin{theorem}[Convergence of PAV algorithm]
    \label{thm:PAV_loc_cv}
    The PAV \Cref{algo:pava_weak_nonneg} based on the update rule $B \mapsto \chi(B)$ with $\chi$ defined in \Cref{equation:ell_q_chi} finds a \emph{local} minimizer of the problem~$(P_p)$.
\end{theorem}

\proof{
The proof is given in \suppref{appendix:proofs}; it comes down to showing that each pooling operation in the PAV \Cref{algo:pava_weak_nonneg} preserves \Cref{theorem:ell_q_local-1} which is a sufficient condition to reach local optimality.
}

\subsection{A Decomposed PAV algorithm (D-PAV)}

\label{subsection:ell_q_oscar}
In this section we propose a new algorithm (D-PAV), based on the  PAV algorithm, which is designed to explore additional local minima.
More precisely, under the additional \Cref{hypothesis:global_mins}, we study necessary conditions on global minimizers of the nonconvex sorted  proximal problem, which naturally give a heuristic to modify the PAV algorithm.
This new hypothesis must be checked for each sorted nonconvex penalty considered.
In the case of $\ell_q$, we prove that it holds as long as the non-increasing  sequence of regularization parameters $(\lambda_i)_i$ is linear (OSCAR-like) or concave.

\begin{hypothesis}
    \label{hypothesis:global_mins}
    The non-increasing  sequence of regularization parameters $(\lambda_i)_i$ is such that, for each block of indices $B := \intervalle{q}{r} \subset \intervalle{1}{p}$ and each sub-block $B_1 := \intervalle{q}{r'}$ with $r' \leq r$,
    \begin{equation}
        \rho^+(T(\bar \lambda_B), \bar \lambda_B) \geq m(\bar \lambda_{B_1}).
    \end{equation}
\end{hypothesis}

\begin{example}[Sorted $\ell_q$]
    \label{example:extension_concave}
        We assume the regularization parameters $(\lambda_i)_i$ write as
       $
            \lambda_i = \Lambda(i) \ ,
       $
        with $\Lambda$ \emph{concave}, non-increasing and non-negative.
        Then, \Cref{hypothesis:global_mins} holds for the $\ell_q$ penalty with $q \in (0,1)$.
        Proof is given in \suppref{proof:concave_ellq}.
\end{example}

\begin{theorem}[Necessary conditions on global minimizers]
    \label{theorem:oscar_ell_q_global_nec}
    Let the regularization parameters $(\lambda_i)_i$ be such that \Cref{hypothesis:global_mins} is satisfied.
    Let $\rvx^p$ be a global minimizer of $(P_p)$. Then, $\rvx^p \in \gS_p$, with $\gS_p$ given by
    \begin{align*}
        \gS_p & := \left\{ (\rvu^{i^\star-1}, \chi(\intervalle{i^\star}{p}),\dots, \chi(\intervalle{i^\star}{p}) ), \rvz^1, \dots , \rvz^p \right\} \ ,\\
        \rvz^i & := (\rvu^{i-1}, 0 , \dots, 0) \quad \forall  i \in \intervalle{1}{p} \ , \\
        \rvu^{i} & \in \gL_{i} \quad \forall  i \in \intervalle{1}{p} \ ,
    \end{align*}
    where $i^\star$ is the largest index such that  there exists $\rvu^{i^\star-1} \in \gL_{i^\star-1}$ with $(\rvu^{i^\star-1}, \chi(\intervalle{i^\star}{p}),\dots, \chi(\intervalle{i^\star}{p} ) )$  feasible.
\end{theorem}

\proof{
The proof of \Cref{theorem:oscar_ell_q_global_nec} is given in \suppref{proof:theorem:oscar_ell_q_global_nec}.
\Cref{hypothesis:global_mins} restricts the set of points where \Cref{theorem:ell_q_local-2} occurs.
This simplifies the setting, as the last block of the solution, which is always a \emph{global} minimizer for its related \emph{scalar} block function, can only be obtained by merging unsorted blocks.
Then, the proof establishes these necessary conditions on \emph{global} minimizers through a geometrical approach. It constructs other local minimizers using conditions outlined in \Cref{theorem:ell_q_local}, and show they can not achieve global optimality.
}

\begin{proposition}\label{propo:PAV_correct_last_block}
Assume the global minimizer of $(P_p)$ is given by $(\rvu^{i^\star-1}, \chi(\intervalle{i^\star}{p}),\dots, \chi(\intervalle{i^\star}{p}) )$ where $\rvu^{i^\star-1} \in \gL_{i^\star-1}$ with $i^\star$ is defined as in \Cref{theorem:oscar_ell_q_global_nec}.
    Then, the PAV \Cref{algo:pava_weak_nonneg}, based on the update rule $B \mapsto \chi(B)$ in~\eqref{equation:ell_q_chi}, returns a solution with the correct last block.
\end{proposition}

\begin{proof}
    Let $(\rvu^{i^\star-1}, \chi(\intervalle{i^\star}{p}),\dots, \chi(\intervalle{i^\star}{p}) )$ be the  global minimizer of $(P_p)$.
    The PAV algorithm returns a \emph{feasible solution}, which is also a local minima of $(P_p)$ (\Cref{thm:PAV_loc_cv}) given by $(\rvv^{j^\star-1}, \chi(\intervalle{j^\star}{p}),\dots, \chi(\intervalle{j^\star}{p} ) )$ with $\rvv^{j^\star-1} \in \gL_{j^\star-1}$.
    By \Cref{theorem:oscar_ell_q_global_nec}, this feasibility ensures that $j^\star \leq i^\star$.
    Now, by contradiction, assume $j^\star < i^\star$.
    Then, consider the maximal block $B = \intervalle{q}{r}$ in $\rvu^{i^\star-1}$ such that $j^\star \in B$.
    Denote $\tilde u$ the value of $\rvu^{i^\star-1}$ on this block $B$.
    By \Cref{theorem:ell_q_local}~\ref{theorem:ell_q_local_point_2}
    \begin{equation}
        \tilde u \leq \chi(\intervalle{j^\star}{r}).
    \end{equation}
    Yet, by \Cref{lemma:pav_blocks}, all the blocks constructed by the PAV algorithm satisfy \Cref{theorem:ell_q_local-1} of \Cref{theorem:ell_q_local}:
    \begin{equation}
        \chi(\intervalle{j^\star}{r}) \leq \chi(\intervalle{j^\star}{p}),
    \end{equation}
    and, by \Cref{theorem:ell_q_local}~\ref{theorem:ell_q_local_point_2}, it also holds
    \begin{equation}
         \chi(\intervalle{j^\star}{p}) \leq \chi(\intervalle{i^\star}{p}).
    \end{equation}
    The previous equations give  $\tilde u \leq \chi(\intervalle{i^\star}{p})$
    which contradicts the feasibility of the global optimizer.
\end{proof}

\paragraph*{The D-PAV algorithm}
From \Cref{prop:variation_scalar}, we know that scalar problems admits at most two local minimizers (see also \Cref{fig:shapes_F}).
The update rule in the PAV algorithm can thus be defined in two distinct ways: using either the threshold $\tau(\lambda)$ or $T(\lambda)$. The results in \Cref{thm:PAV_loc_cv} and \Cref{propo:PAV_correct_last_block} are derived considering $\tau(\lambda)$, specifically the $\chi$ update defined in~\eqref{equation:ell_q_chi}, which updates
 the blocks with the \emph{largest local minimizer} of the corresponding block scalar problem. %
This may seem counter-intuitive and one may find more natural to %
update each block with the global minimizer of block scalar problem, that is using $T(\lambda)$.
However, this would force too many entries to $0$ (see \Cref{fig:counter_ex}).

While we have shown that the PAV algorithm with the $\chi$ update rule finds a local minima and correctly identifies the last block of the global optimizer whenever this one is given by $\chi$ and not by $0$, it does not explore local minimizers with a last block of $0$, denoted as $\rvz$ in \Cref{theorem:oscar_ell_q_global_nec} (\Cref{fig:counter_ex}).

This motivates our proposed \textit{decomposed PAV algorithm} (D-PAV, \Cref{algo:sorted_ell_q}) to explicitly consider these local minimizers.
Specifically, as we compute the PAV solution, we store the intermediate solutions of size $k \leq p$, complete them with zeros at the end of the vector, and finally keep the one leading to the smallest objective value.

\begin{figure}
    \centering
    \includegraphics[width=\linewidth]{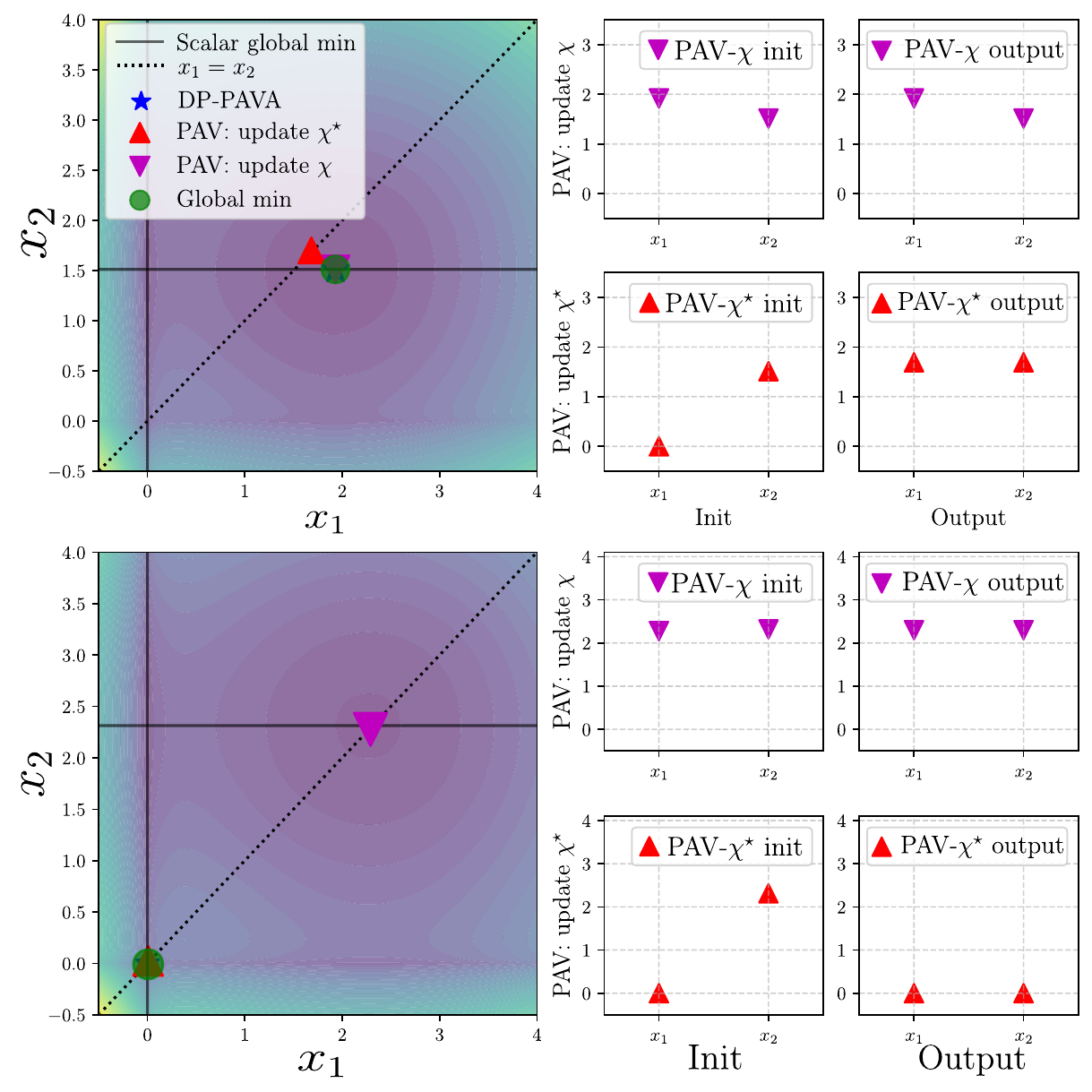}
    \caption{2D sorted $\ell_q$ proximal computation. Comparison of PAV with $\chi$ update rule (the largest local minimizer of the scalar block problem, threshold $\tau(\lambda)$), PAV with $\chi^\star$ update rule (the global minimizer of the scalar block problem, threshold $T(\lambda)$) and D-PAV.
    Top: PAV with $\chi^\star$ update rule leads to unwanted merging due to too many $0$ in the initialization.
    Bottom: While PAV with the $\chi$ update rule recovers a local minima, it does not consider solutions with a last block of $0$.
    In both cases, the proposed D-PAV finds the global minimizer.
    In all cases, the initialization corresponds to applying $\chi$ or $\chi^\star$ pointwise, i.e., an unconstrained local minimizer.
    }
    \label{fig:counter_ex}
\end{figure}

\begin{algorithm}[t]
  \caption{D-PAV for sorted nonconvex penalties}
  \label{algo:sorted_ell_q}
  \begin{algorithmic}
  \STATE{Data $\rvy \in \sR^p$, regularization sequence $(\lambda_i)_{i \in \intervalle{1}{p}}$ }
  \STATE{ $\rvy \gets \vert \rvy \vert_\downarrow$}
  \FOR{$k \in \llbracket 1, p \rrbracket $}
  \STATE{$\rvx^k_{1:k} \gets $ PAV $(\rvy_{[k]}, \bm{\lambda}_{[k]})$ }
  \STATE{$\rvx^k_{i+1:p} \gets 0$}
  \IF{$F(\rvx^k) < F(\rvx)$}
  \STATE{ $\rvx \leftarrow \rvx^k$}
  \ENDIF
  \ENDFOR
  \RETURN $\rvx$
  \end{algorithmic}
\end{algorithm}

\begin{remark}
  Our algorithm D-PAV may fail to find the global minimizer because, for each subproblem of size $k < p$, it only considers one local minimizer among all possible ones, the one determined by $\mathrm{PAV}(\rvy_{[k]}, \bm{\lambda}_{[k]})$.
  In particular, the PAV algorithm never considers local minimizers in the shape of \Cref{theorem:ell_q_local-2}.
  However, we argue that the instances $(\rvy, \bm{\lambda})$ which lead to this undesirable behavior are very rare.
  Besides, even in such adversarial cases, our experiments always find out that D-PAV succesfully recovers the global minimizer.
  See \appref{appendix:expe_limits} for more details on the numerics.
\end{remark}

\section{Experiments}
\label{section:expe}
The code to reproduce all the experiments will be released on GitHub upon acceptance.
It relies on \texttt{sklearn} \cite{pedregosa2011scikit} and \texttt{skglm} \cite{bertrand2022beyond}.

\subsection{Denoising}\label{sub:denoising}
In this experiment, we demonstrate the property that, generalizing from classical $\ell_1$ versus nonconvex penalties, %
\emph{sorted} nonconvex penalties lead to estimators with less amplitude bias than their convex counterpart SLOPE.

The setup is the following:
we generate a true vector $\rvx^* \in \sR^{28}$ with 4 equal-sized clusters of coefficients, on which it takes the values $7, -5, 3$ and $-1$.
The clusters are taken contiguous simply to visualize them better in \Cref{fig:denoising_solutions}, but we stress that this does not help the algorithms: their performance would be the same if we shuffled the vector to be recovered.
We add i.i.d Gaussian noise of standard deviation $0.3$ to $\rvx$ to create $1000$ samples of noisy vector $\rvx$.

To denoise $\rvx$, we apply the proximal operators of SLOPE, sorted MCP and sorted $\ell_{1/2}$ respectively.
The sequence $(\lambda_i)_i$ is selected\footnote{for $\ell_{1/2}$ we use $r (28 - i)^{1.5}$ as it empirically turned out to yield better performance.} as $\lambda_i = r (28 - i)$, with $r > 0$ controlling the strength.
For each of these three sorted penalties, we compute denoised versions of $\rvx$  for 100 geometrically-spaced values of $r$.
For each resulting denoised vector $\hat \rvx$, we define the normalized error as
$\Vert \hat \rvx - \rvx^* \Vert / \Vert \rvx^* \Vert$.
To measure the cluster recovery performance, we define the binary matrix $M(\rvx) \in \mathbb{R}^{28 \times 28}$ as having entry $i, j$ equal to 1 if $|x_i| = |x_j|$ (i.e. $i$ and $j$ are in the same cluster in $\rvx$) and 0 otherwise.
The reported F1 score is then the F1 score between $M(\hat \rvx)$ and $M(\rvx^*)$, i.e. the F1 score of the classification task ``predict if two entries of the vector belong to the same cluster''.
An F1 score of 1 means that $\hat \rvx$ perfectly recovers the clusters of $\rvx^*$.

In \Cref{fig:denoising_rmse_f1} we report the variations of F1 score and error as a function of penalty strength $r$.
For each penalty, we define the optimal  $r$ as the lowest $r$ for which the F1 score is superior to 0.75 (vertical dashed lines).
One can see that, for an equivalent F1 score, sorted nonconvex  penalty reach a lower error.
Denoised vectors for the choice of $r$ that gives a F1 score equals to 0.75 are displayed in \Cref{fig:denoising_solutions} for a single repetition only which confirms that, for an equivalent cluster recovery performance (as measured by F1 score), the nonconvex penalties reach a much lower error $\Vert \hat x - x^* \Vert$.

\begin{figure}[t]
    \centering
    \includegraphics[width=\linewidth]{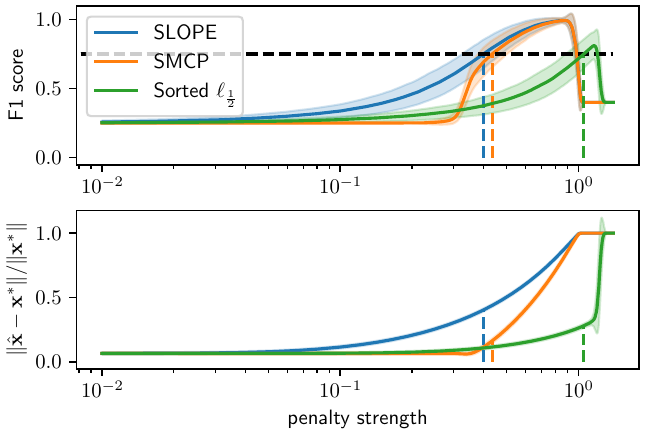}
    \caption{
        \emph{Top}: F1 score (averaged on the $1000$ samples, with standard deviation) for cluster recovery, as a function of the penalty strength.
        \emph{Bottom}: Normalized estimation error (averaged on the $1000$ samples, with standard deviation), as a function of the penalty strength.
        For each penalty, the dashed line is set at the F1 score equal to $0.75$ and displays a tradeoff between small estimation error and high F1 score.
        }
    \label{fig:denoising_rmse_f1}
\end{figure}

\subsection{Regression}

We now turn to the setting of \Cref{pb:composite}, with a least squares datafit $g(\rvx) =  \frac{1}{2} \Vert \mA \rvx - \rvb \Vert^2$.
The setup is the following.
The matrix $\mA$ is in $\sR^{50 \times 100}$, it is random Gaussian with Toeplitz-structured covariance $\Sigma$: $\Sigma_{ij} = (0.98^{|i - j|})_{i,j}$.
The ground truth $\rvx^*$ has entries from 1 to 10 and 50 to 55 equal to 5, the rest being 0.
The target $\rvb$ is $\mA \rvx^* + \bm{\epsilon}$ where $\bm{\epsilon}$ has independent Gaussian entries, scaled so that the SNR $\norm{\mA \rvx^*} / \norm{\bm{\epsilon}}$ is equal to $7$.
All penalties use the following sequence of regularization parameters $ r \, (1, 2^{\frac 2 3} - 1, \dots, 100^{\frac 2 3} - 99^\frac{2 3} )$ in the spirit of the quasi-spherical OSCAR sequence \cite{nomura2020exact}, where $r > 0$ a scaling parameter to control penalty strength.

In \Cref{fig:regression_stem} we report regression results for four values of penalization strength.
It highlights the bias of SLOPE, which creates many false positives and still underestimates the ground truth strength as the penalty strength diminishes.
On the contrary, sorted $\ell_{1/2}$ and MCP have few false positives and better identify the magnitude of the true nonzero coefficients.
\begin{figure*}[!t]
    \centering
    \hspace{3mm} \includegraphics[width=0.6\linewidth]{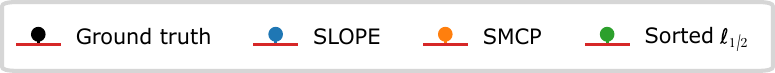}
    \includegraphics[width=\linewidth]{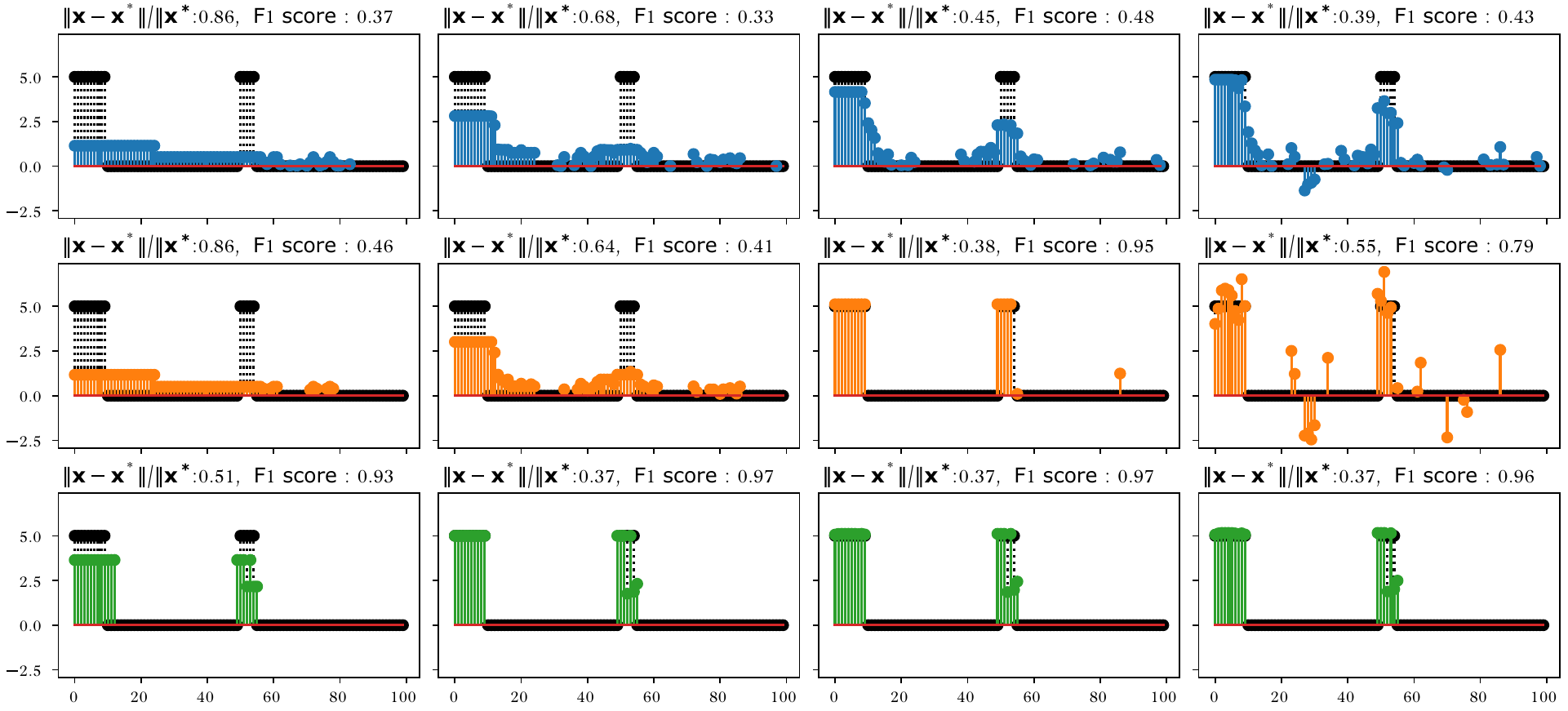}
    \caption{Solutions of $\varPsi$-penalized regression for SLOPE (blue), sorted MCP (yellow) and sorted $\ell_{1/2}$ (green), for decreasing penalty strengths (from left column to right).
    The ground truth is in black; SLOPE has more false positive and bias.
        }
    \label{fig:regression_stem}
\end{figure*}

\subsection{Real Data}

We illustrate here the solution paths on the \textit{diabetes} data set \cite{lsa2004}:
For these numerics, we compare various choices of penalties on a least square regression problem: $b \in \sR^{442}$ measures the disease progression for $n =442$ diabetes patients while the columns of $A \in \sR^{442 \times 10}$ describe $d = 10$ relevant features (age, sex, blood pressure, etc.).
We focus on a low-dimensional problem so as to visualize the solution paths.
For the sorted penalties (i.e. SLOPE, SMCP, Sorted $\ell_{1/2}$), we take $\lambda =  r (1, 2^{1/4} - 1, \dots , 11^{\frac 1 4} - 10^{\frac 1 4})$ in the spirit of the quasi-spherical OSCAR sequence \cite{nomura2020exact}, and $r$ varies along the path.
For the MCP and the sorted MCP penalties, the parameter $\gamma$ is equal to $1$.
Results are displayed on \Cref{fig:sol_paths}.
We observe that all sorted penalties indeed cluster features (so does SLOPE when compared to LASSO) and decreases amplitude bias (so do MCP and $\ell_{1/2}$ when compared to LASSO).

\begin{figure*}
    \centering
    \subfloat[Non sorted penalties]{\includegraphics[width=0.48\linewidth, trim={16 10 45 40},clip]{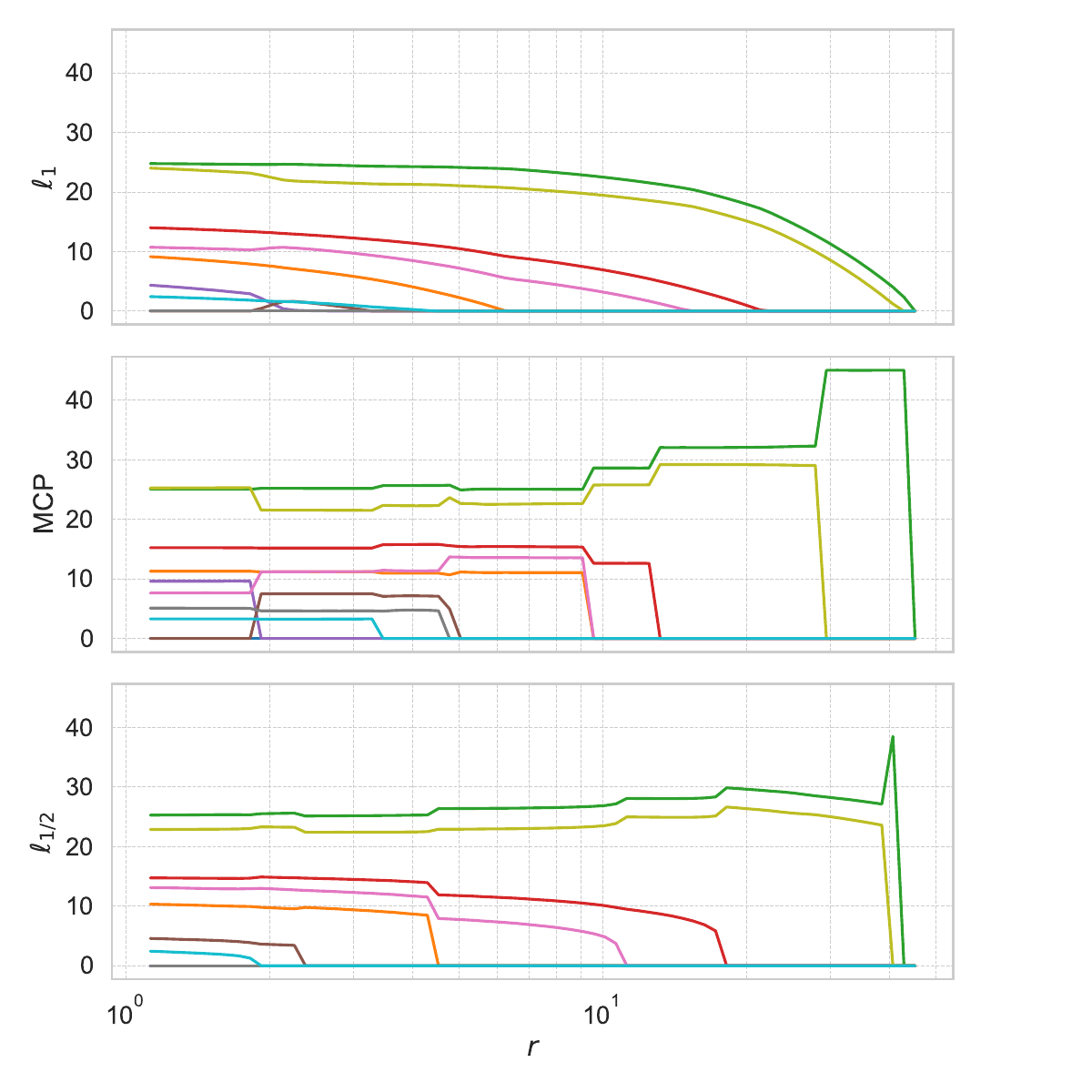}
        \label{fig:lasso_mcp_diabetes}} %
    \subfloat[Sorted penalties]{\includegraphics[width=0.48\linewidth, trim={16 10 45 30},clip]{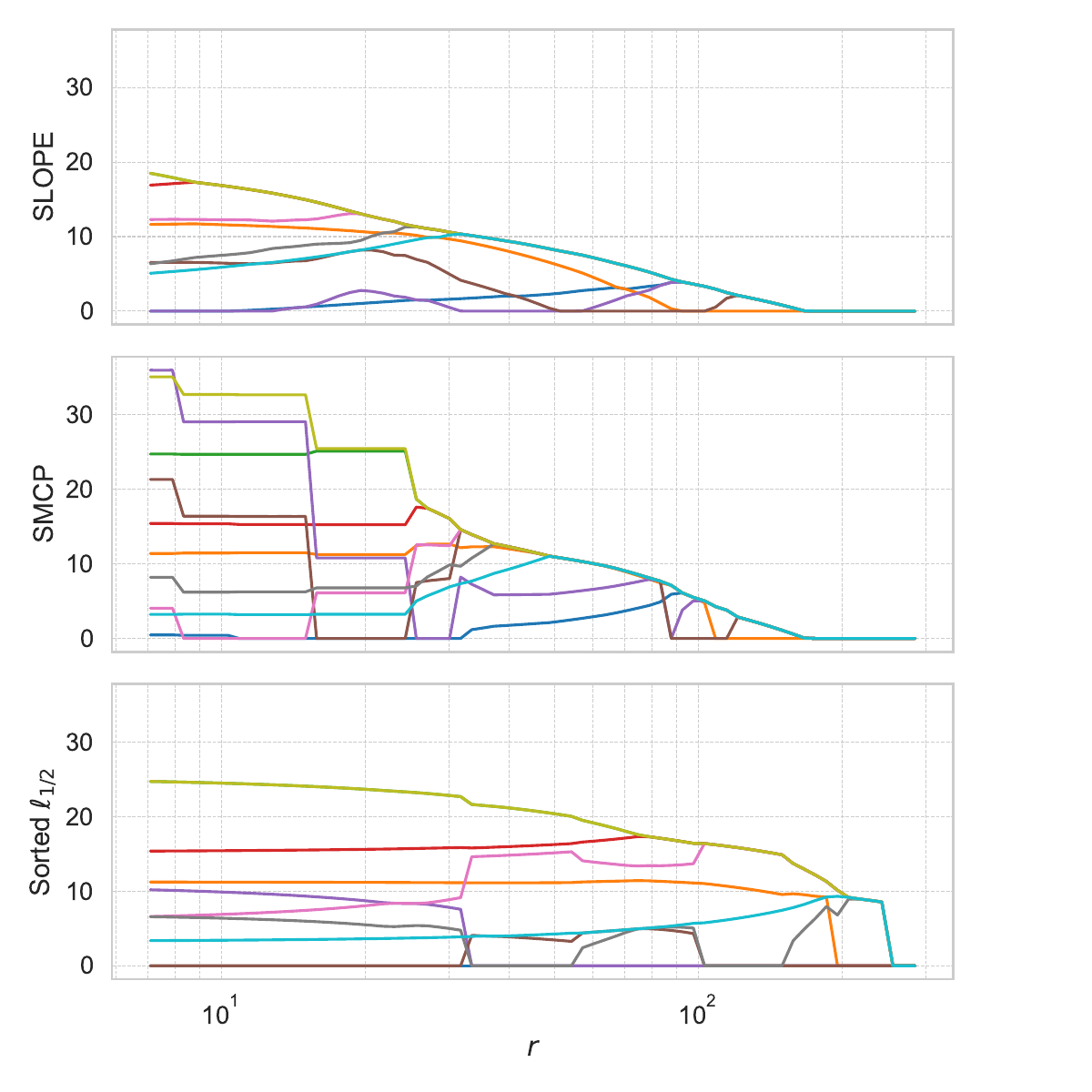}
      \label{fig:slope_diabetes}}
    \caption{Solution paths for convex, nonconvex, sorted and non sorted penalties.}
    \label{fig:sol_paths}
\end{figure*}

\section{Conclusion}

Sorted nonconvex penalties generalize the Sorted $\ell_1$ norm penalty: they benefit from both the sparsity-enhancing and clustering properties of SLOPE but promote solutions with less amplitude bias.
We have proposed a unified approach to deal with regularized regression problems using such penalties by providing efficient algorithms to compute, exactly or approximately, their proximal operator, which paves the way for a more general use in practical cases.

\appendices

\section{Details on Clustering Properties of Sorted Nonconvex Penalties}
\label{app:details_clustering}
For every solution $\rvx^*$ of $\min_\rvx g(\rvx) + \Psi(\rvx)$, there exists a value of $\tau >0$ such that $\rvx^*$ also solves
\begin{align}
    & \min g(\rvx)  \quad  \text{s.t. } \Psi(\rvx) \leq \tau \, .
\end{align}
Hence, solving the the regularized problem can be seen as finding the smallest sublevel set of $g$ that intersect the ball $\{\rvx : \varPsi(\rvx) \leq \tau \}$, the solution being the corresponding intersection point.
To get intuition on how the clustering operates using sorted penalties, \Cref{fig:balls_slope} displays the projection on the unit ball for various choices of penalties:
\begin{itemize}
    \item Points of non-differentiability of the penalty attract the solution towards corners of the unit ball.
    They enforce sparsity for the $\ell_1$ ball and both sparsity and clustering for the SLOPE and sorted $\ell_{1/2}$ unit balls.
    \item For the $\ell_{1/2}$ unit ball, due to the nonconvexity of the ball, the corners associated with clustering are more or less attractive depending on the choice of the regularization sequence $(\lambda_i)_i$.
\end{itemize}
We also plot the level lines of sorted MCP.
We observe that for $\tau$ large enough in the constrained problem, the sublevel set is the whole space  (the penalty saturates), the projection is the identity and there is no longer any clustering. Yet, for smaller $\tau$, we still recover some corners, which are associated with clustering.

\begin{figure}[t]
\centering
	\begin{tikzpicture}
		\node (A) at (0,0) {\includegraphics[width=0.32\linewidth,trim={0.9cm 0.7cm 1cm 1.2cm}, clip]{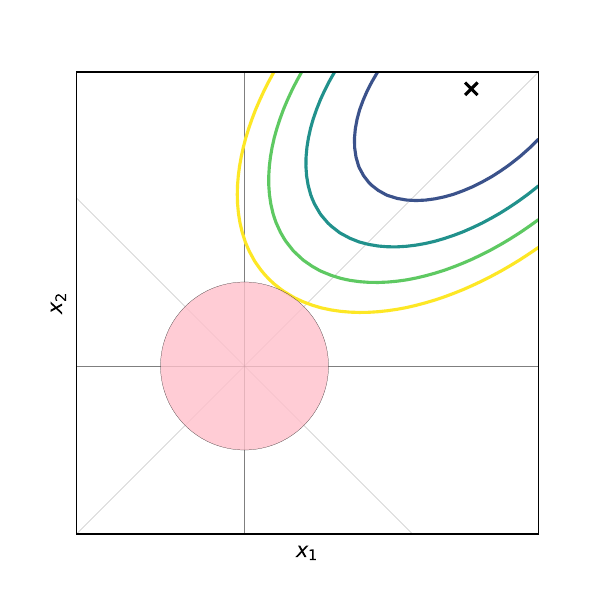}
        \label{fig:l2_ball}};
        \node[anchor=north] at (A.south) {$\ell_2$};
        \node (B) at (0.33\linewidth,0) {\includegraphics[width=0.32\linewidth,trim={0.9cm 0.7cm 1cm 1.2cm}, clip]{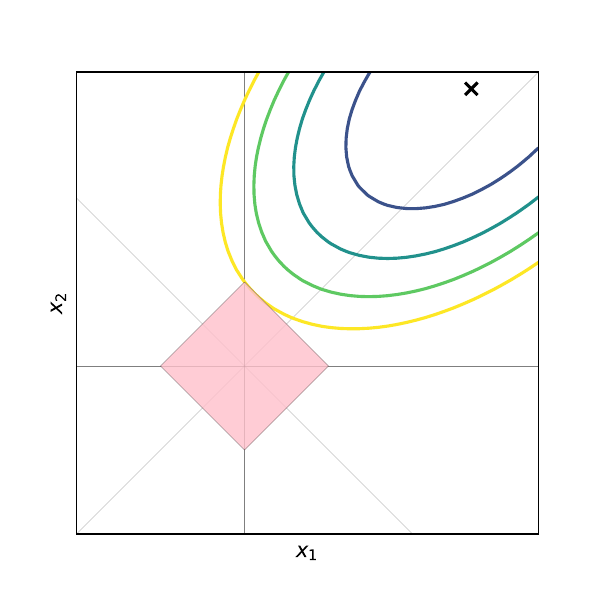}
        \label{fig:l2_ball}};
        \node[anchor=north] at (B.south) {$\ell_1$};
        \node (C) at (0.66\linewidth,0) {\includegraphics[width=0.32\linewidth,trim={0.9cm 0.7cm 1cm 1.2cm}, clip]{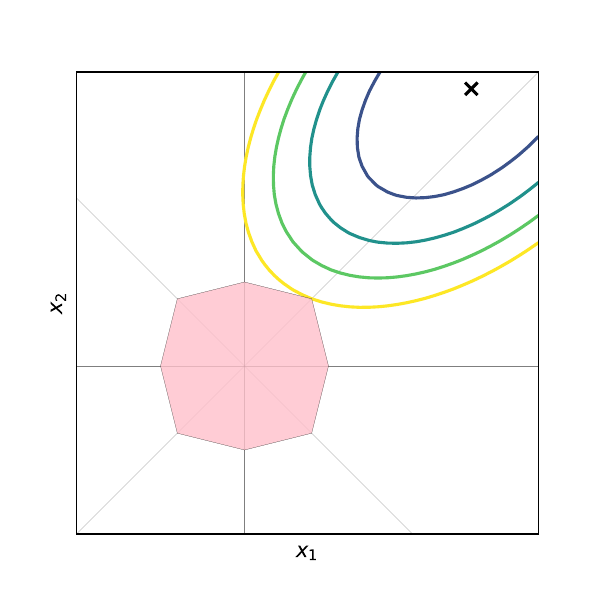}
        \label{fig:l2_ball}};
        \node[anchor=north] at (C.south) {SLOPE};
        \node (D) at (0,-0.43\linewidth) {\includegraphics[width=0.32\linewidth,trim={0.9cm 0.7cm 1cm 1.2cm}, clip]{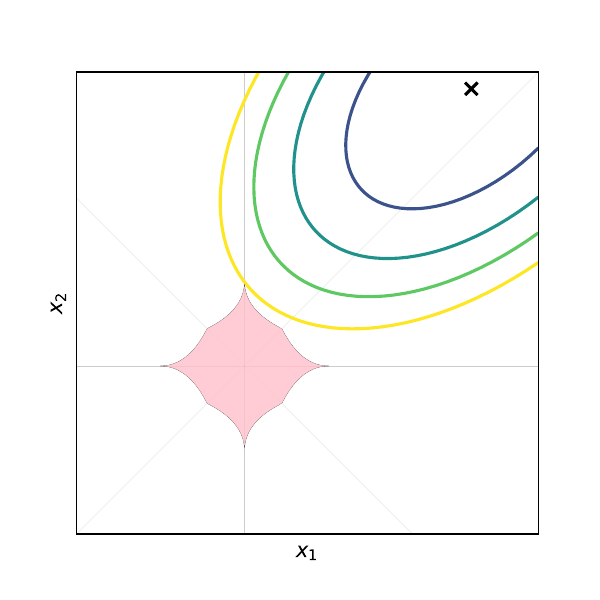}
        \label{fig:l2_ball}};
        \node[anchor=north, align=left] at (D.south) {Sort. $\ell_{1/2}$, \\ $\lambda_2 = \frac{\lambda_1}{2}$};
        \node (E) at (0.33\linewidth,-0.43\linewidth) {\includegraphics[width=0.32\linewidth,trim={0.9cm 0.7cm 1cm 1.2cm}, clip]{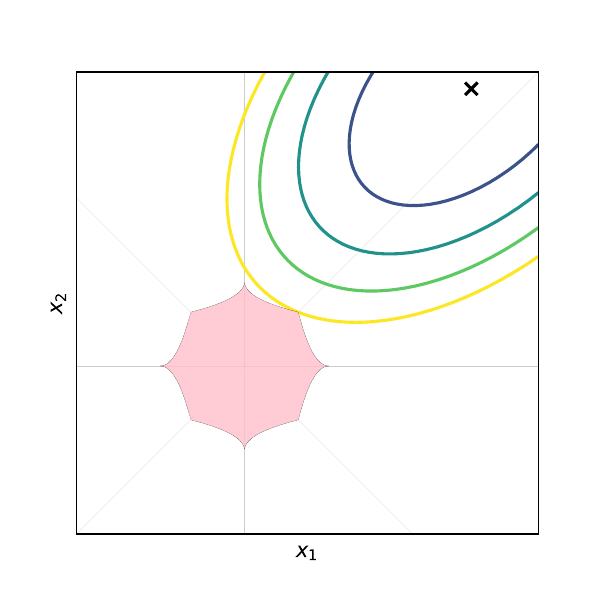}
        \label{fig:l2_ball}};
        \node[anchor=north, align=left] at (E.south) {Sort. $\ell_{1/2}$,  \\ $\lambda_2 = \frac{\lambda_1}{4}$};
        \node (F) at (0.66\linewidth,-0.43\linewidth) {\includegraphics[width=0.32\linewidth,height=0.32\linewidth,trim={0.9cm 0.7cm 1cm 1.2cm}, clip]{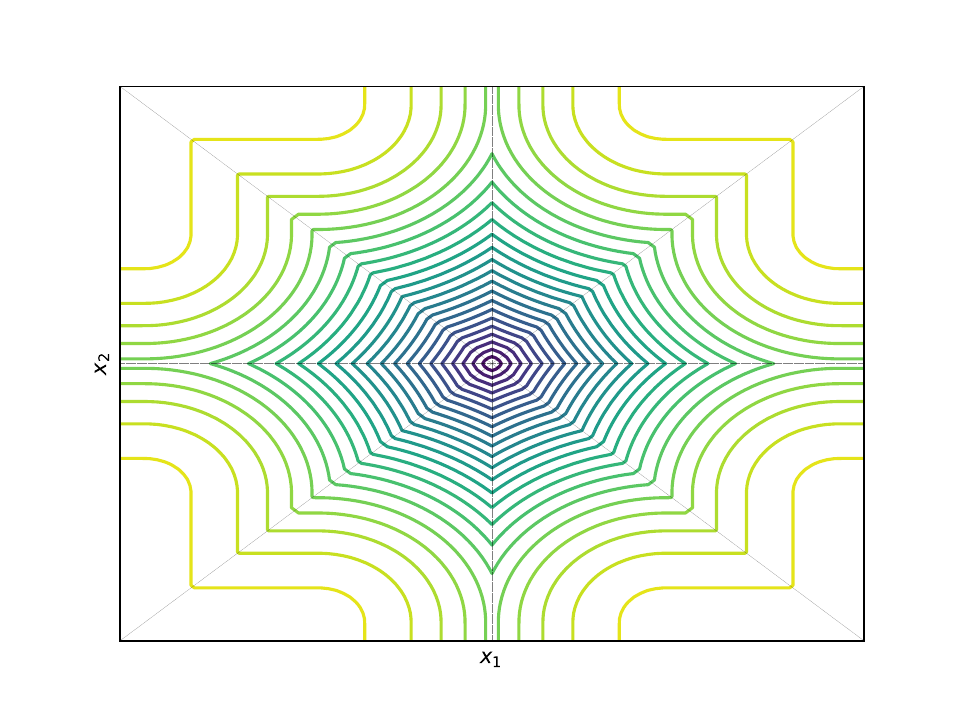}
        \label{fig:l2_ball}};
        \node[anchor=north, align=left] at (F.south) {S-MCP,   $\gamma = 1.4$ \\ $\lambda_1 = 3, \lambda_2 = 2.5$};
	\end{tikzpicture}
    \caption{
    Sparse regression in its constrained form.
    We display the level lines of the least squares datafit and the sublevel set of various penalties (red).
    \emph{Sparsity}: the non-differentiability of $\ell_1$ attracts the solution towards the corners of the ball.
    \emph{Structure}: SLOPE also attracts solutions towards points with equal coefficients. These corners are also present in sorted $\ell_{1/2}$ and their attractivity depends on the choice of the regularization sequence.}
    \label{fig:balls_slope}
\end{figure}

\section{Pool Adjacent Violators algorithm}
\label{appendix:pava}

The PAV algorithm is recalled in \Cref{algo:pava_weak_nonneg}.
In the following, a block is a set of consecutive indices of $\intervalle{1}{p}$ on which the components of the iterate $\rvx$ are equal.
At each iteration of the algorithm, a block $B_0 := \intervalle{q}{r}$ is taken as the \emph{working block}.
Likewise, $B_- := \intervalle{...}{q-1}$ and $B_+ := \intervalle{r + 1}{\dots}$ respectively denote the blocks before and after $B_0$, potentially empty.
We denote $x_B$ as the common value of the components of $\rvx$ on the block $B$.
We introduce two functions that operate on blocks: $\mathrm{Pred}(B)$ (resp. $\mathrm{Succ}(B)$) returns the block just before (resp. after) the block $B$ if it exists, otherwise, it returns the empty set.

\begin{algorithm}
    \caption{PAV Algorithm}\label{algo:pava_weak_nonneg}
    \begin{algorithmic}
        \STATE \textbf{Input:} data $\rvy \in \sR^p$
        \STATE $\rvx \gets \left( \chi(\{1\}),\dots,\chi(\{p\}) \right)$ \COMMENT{Initialize with unconstrained solution}
        \STATE $B_- \gets \emptyset$
        \STATE $B_0$ is the block such that $1 \in B_0$, and $B_+ \gets \mathrm{Succ} (B_0)$
        \WHILE{$B_+ \neq \emptyset$}
            \IF{$x_{B_0} \geq x_{B_+}$}
                \STATE $B_0 \gets B_+$, $ B_- \gets B_0 $, $B_+  \gets \mathrm{Succ}(B_+)$ \COMMENT{If blocks are ordered, move forward}
            \ELSE
                \STATE Update $\rvx$ on $B_0 \cup B_+$: $x_{B_0 \cup B_+} \gets \chi(B_0 \cup B_+)$ \COMMENT{If not, pooling operation}
                \STATE $B_0 \gets B_0\cup B_+$, $B_+ \gets \mathrm{Succ}(B_0)$
                \WHILE{$( x_{B_-} \leq  x_{B_0} ) \land ( B_- \neq \emptyset )$}
                    \STATE Update $\rvx$ on $B_- \cup B_0$: $x_{B_- \cup B_0} \gets \chi(B_- \cup B_0)$ \COMMENT{Backward pass to ensure correct ordering after pooling}
                    \STATE $B_0 \gets B_- \cup B_0$, $B_- \gets \mathrm{Pred}(B_-)$
                \ENDWHILE
            \ENDIF
        \ENDWHILE
        \RETURN $(\rvx)_+$
    \end{algorithmic}
\end{algorithm}
\section{Proof of \Cref{theorem:ell_q_local}}
\label{proof:min_local_scalar}
\begin{proof}
    We prove each item independently.
    \begin{itemize}
    	\item \textit{Proof of \Cref{prop:variations_item_1}.}
    Define $m(\lambda)$ as
    \begin{equation}\label{eq:m_lamb}
        m(\lambda) := \inf \left\{ z \in \sR_{+*} \setminus Z_\psi, \psi_0''(z) \geq \frac{-1}{\lambda} \right\} .
    \end{equation}
    Since $\psi''$ is increasing, for all $z \in \sR_{+*} \setminus Z_\psi$ such that $z < m(\lambda)$, $F''(z) < 0$.
    Similarly, for all $z \in \sR_{+*} \setminus Z_\psi$ such that $z \geq m(\lambda)$, $F''(z) \geq0$.
    Hence, on each segment where $\psi''$ is continuous, $F$ is strictly concave for $z < m(\lambda)$ and convex for $z \geq m(\lambda)$.
    Moreover, as $\psi$ is continuously differentiable on $\sR_{+*}$ the result extends for all $z \in \sR_{+*}$.
    Finally, by continuity of $F$, the property holds on $\sR_+$.

    \item \textit{Proof of \Cref{prop:variations_item_2}.} The nonzero minimizer, if it exists, must necessarily lie in $[m(\lambda), +\infty[$, where  $F$ is convex.
    It exists if $F'$ has a root on $[m(\lambda), +\infty[$, \emph{i.e.} if $F'(m(\lambda))<0$ which is equivalent to $y >m(\lambda)+  \lambda \psi_0'(m(\lambda)) =: \tau(\lambda)$. In addition $0$ is a local minimizer of $F$ on $\R_+$ if $\lim_{z \to 0} F'(z) \geq 0$, that is $y \leq \lambda \psi_0'(0)$.
 Then,   $\rho^+$ is continuous with respect to $y$ from the bijection of $z \mapsto z + \lambda \psi'$ which is continuous and strictly monotone on $[m(\lambda), +\infty)$.
    Then, we show that $\rho^+$ is non-decreasing with $y$.
    Indeed, consider $y_1 > y_2 >0 $ and $\lambda>0$.
    Assume, by contradiction, $\rho^+(y_2;\lambda) > \rho^+(y_1;\lambda)$ (note that both are greater than $\tau(\lambda)$ by definition, hence than $m(\lambda)$).
    Then, as $F'_{y_2} : z \mapsto z-y_2+ \lambda \psi_0'(z)$ is increasing on $[m(\lambda), +\infty[$, and as $F'_{y_2}(\rho^+(y_2;\lambda)) = 0$, it holds $F'_{y_2}(\rho^+(y_1;\lambda)) < 0$.
    Yet, by definition of $\rho^+$, it holds $F'_{y_1}(\rho^+(y_1;\lambda)) = 0$, which leads to the following contradiction
    \begin{align*}
        0 & = \rho^+(y_1;\lambda) -y_1 +  \lambda \psi_0'(\rho^+(y_1 ; \lambda)) \\
        & < \rho^+(y_1;\lambda) -y_2 +  \lambda \psi_0'(\rho^+(y_1;\lambda)) \\
        & = F'_{y_2}(\rho^+(y_1;\lambda)) < 0\ .
    \end{align*}

    Next, we show that $\rho^+$ is non-increasing with $\lambda$.
    Indeed, consider $\lambda_1 > \lambda_2 >0 $ and $y>0$.
    Again, assume by contradiction $\rho^+(y;\lambda_1) > \rho^+(y;\lambda_2)$.
    Then, as $F'_{\lambda_2} : z \mapsto z-y+ \lambda_2 \psi_0'(z)$ is increasing on $[m(\lambda_2), +\infty[$ and as $F'_{\lambda_2}(\rho^+(y;\lambda_2)) = 0$, it holds $F'_{\lambda_2}(\rho^+(y;\lambda_1)  ) >0$.
    As previously, we obtain the following contradiction
    \begin{align*}
        0 & = \rho^+(y;\lambda_1) -y +  \lambda_1 \psi_0'(\rho^+(y;\lambda_1)) \\
        & >  \rho^+(y;\lambda_1) -y +  \lambda_2 \psi_0'(\rho^+(y;\lambda_1)) \\
        & = F'_{\lambda_2}(\rho^+(y;\lambda_1)  ) > 0 \ ,
    \end{align*}
    as $\psi_0' >0$.
    \item \textit{Proof of \Cref{prop:variations_item_3}.} Since the prox is a monotone operator, if $\{0, \rho^+(y;\lambda) \} \in \prox_{\lambda \psi_0} (y)$ for some $y \in \sR_+$ then, for all $\tilde{y} < y$, $\prox_{\lambda \psi_0} (\tilde{y}) = \{0\}$ and for all $\tilde y > y$,  $\prox_{\lambda \psi_0} (\tilde y) =  \{ \rho^+(\tilde y; \lambda) \}$.
    Hence to prove the existence of the threshold $T(\lambda)$, it is sufficient to exhibit two values $y_0 < y_+$ such that $\prox_{\lambda \psi_0}(y_0) = \{ 0 \}$ and $\prox_{\lambda \psi_0}(y_+) = \{ \rho^+(y_+;\lambda) \}$. The first one, $y_0$, is trivial from \Cref{prop:variations_item_2} where we get that $0$ is the unique local (and thus the global) minimizer of $F_y$ for all $y < \tau(\lambda)$. Regarding $y_+$, if it exists, it should be such that $y_+ \geq \tau(\lambda) \geq m(\lambda)$. It thus belongs to the convex region of $F_{y_+}$ (\Cref{prop:variations_item_1}), and by definition of $\rho^+$ (minimizer of the convex region) we have $F_{y_+}(y_+) \geq F_{y_+}(\rho^+(y_+;\lambda))$. To complete the proof we will  show that $F_{y_+}(0) > F_{y_+}(y_+)$, or equivalently $0 > -\frac12 y_+^2 + \lambda \psi_0'(y_+) := G(y_+)$. Since $\psi_0$ is concave and continuously differentiable on $\sR_{+*}$, it lies below its tangents, thus its growth is at most linear.  Hence, $G(y) \to_{y \to +\infty}  -\infty$ which proves the existence of such an $y_+$. Finally the fact that $T(\lambda) \in [\tau(\lambda,\lambda\psi_0'(0)]$ is direct from \Cref{prop:variations_item_2}.

    \qedhere
    \end{itemize}
\end{proof}
\section{MM experiment}
\label{appendix:mm}

For the experimental setup of \Cref{fig:mm}, the matrix $\rmA$ is in $\sR^{150 \times 50}$, it is random Gaussian with Toeplitz-structure covariance equal to $(0.4^{|i-j|})_{i,j}$.
We denote $n$ the number of samples and $d$ the number of features.
The ground truth $\rvx^\star$ has a sparsity of $60 \%$ (i.e. only $40 \%$ of non-zero entries) and is made of two clusters of equal coefficients, respectively equal to $+0.5$ and $-0.5$.
The penalty used is sorted MCP with parameter $\gamma$ equal to $1.1$ and linearly decreasing regularization strength: $\lambda_i = \alpha \times \frac{d-i}{d}$ for $i \in \llbracket 1, d \rrbracket$ where $\alpha$ is determined by grid-search.
We consider, using this synthetic dataset, two different settings:
\begin{enumerate}
    \item Regression: the target $\rvb$ is defined as $\rvb = \rmA \rvx^\star + \bm{\epsilon}$ where $\bm{\epsilon}$ is a vector of white Gaussian noise such that the signal-to-noise ratio is $10$.
    The datafit used is a \emph{least squares datafit} $g(\rvx) = \frac{1}{2} \norm{\rmA \rvx - \rvb}^2$.
    \item  Classification: The target $\rvb$ is equal to $\sign( \rmA \rvx^\star)$ except for $10 \%$ of the samples for which the sign of $b_i$ is flipped.
    The datafit used is a \emph{logistic regression datafit} $g(\rvx) = \frac{1}{n} \sum_{i=1}^n \log \left(1 + e^{-b_i(\rmA \rvx)_i } \right)$.
\end{enumerate}
The algorithms are stopped when the loss difference between two successive iterations is smaller than a tolerance fixed at $10^{-5}$.
We consider a MM algorithm with one step of prox-grad at each MM iteration (defined in \Cref{eqn:mm_iter}).

\section{Building Counter-Examples to D-PAV?}
\label{appendix:expe_limits}

\Cref{theorem:oscar_ell_q_global_nec} defines a subset of local minimizers which contains the global minimizer.
While the D-PAV algorithm explores this subset, there may exist some setting where the algorithm misses some candidates and fails to recover the global minimizer.
From the form of the solution, we know that we recover at least  the correct last block.
Yet, on the first part of the solution, one may be able to find other blocks structure by merging correctly ordered blocks (\Cref{theorem:ell_q_local-2}).
The following experiments show that the combination of $(\rvy, \bm{\lambda})$ such that this setting may occur are very rare.
Moreover, even by purposely constructing such adversarial examples, we cannot find a setting where D-PAV fails.
The following experiments are done for the penalty Sorted $\ell_{\frac 1 2}$.

\textit{D-PAV versus exhaustive search.}  To benchmark our D-PAV approach, we compare it with the bruteforce approach where we compute all the possible partitions of the vector into blocks. We take $\rvy \in \sR^{10}$: each point $y_i$ is such that $y_i = T(\lambda_i) + \epsilon$ with $\epsilon \sim \mathcal N(-0.3,1)$. We compute the value of the blocks using $\chi$. We test on $10$ different seeds for which D-PAV always recovers the solution.
 We also compare ourselves for $\rvy \in \sR^{100}$ with the Sequential Least Squares Programming implemented in the \texttt{scipy} library: we run on $100$ initial points and keep the better solution. Once again, D-PAV systematically beats the solution from \texttt{scipy}.

\textit{Adversarial settings.}
The D-PAV algorithm could fail in a setup where it does not merge sorted blocks $\chi(B_1) \geq \chi(B_2)$ while $\chi(B_1 \cup B_2)$ could have achieved a better objective value.
To explore the rate of occurrence of such a setup, we perform the following experiment.
We consider the $\ell_{1/2}$ penalty.
We fix a value for $\bar \lambda_{B_2} := 1$.
We test five values for $\bar y_{B_2} \in \{ 5,10,15,20,25 \} \times \tau(\bar \lambda_{B_2} )$.
We consider a ratio $ t$ such that $|B_2| = t |B_1\cup B_2|$ and we test $4$ values for $t$ in $[0.3,0.9]$.
We are looking for a set of points $(\bar y_{B_1}, \bar \lambda_{B_1})$ such that $\bar y_{B_1} \geq \bar y_{B_2}$, $\bar \lambda_{B_1} \geq \bar \lambda_{B_2}$ and $\chi(B_1) \geq \chi(B_2) \geq \chi(B_1 \cup B_2)$.
To do so, we explore a grid of $\lambda \in [50,200]$ and $y \in [0,60]$.
For each choice of $t$ and $\bar y_{B_2}$, \Cref{fig:zone_merging} plots points $(\bar y_{B_1}, \bar \lambda_{B_1})$ for which D-PAV could \emph{theorically fail}.

Next, for each choice of $\bar y_{B_2}$ and $t$, we randomly pick a problematic couple  $(\bar y_{B_1}, \bar \lambda_{B_1})$ (corresponding to a blue point in \Cref{fig:zone_merging}) and we build a \emph{candidate} counter-example in $\sR^{22}$ such that
\begin{align*}
    \mathbf y & = (\bar y_{B_1}, \underbrace{\bar y_{B_1} , \dots, \bar y_{B_1}}_{(1-t) \times 20} ,\underbrace{\bar y_{B_2} , \dots, \bar y_{B_2}}_{ t \times 20}, \frac{\bar y_{B_2} }{2}) , \\
    \mathbf \lambda & = (\bar \lambda_{B_1} + 0.1, \underbrace{\bar \lambda_{B_1} , \dots, \bar \lambda_{B_1}}_{(1-t) \times 20} ,\underbrace{\bar \lambda_{B_2} , \dots, \bar \lambda_{B_2}}_{ t \times 20}, 0).
\end{align*}
Our D-PAV algorithm systematically yields better solutions than the ones obtained with the merging described above.

\begin{figure}[t]
    \centering
    \includegraphics[width=\columnwidth]{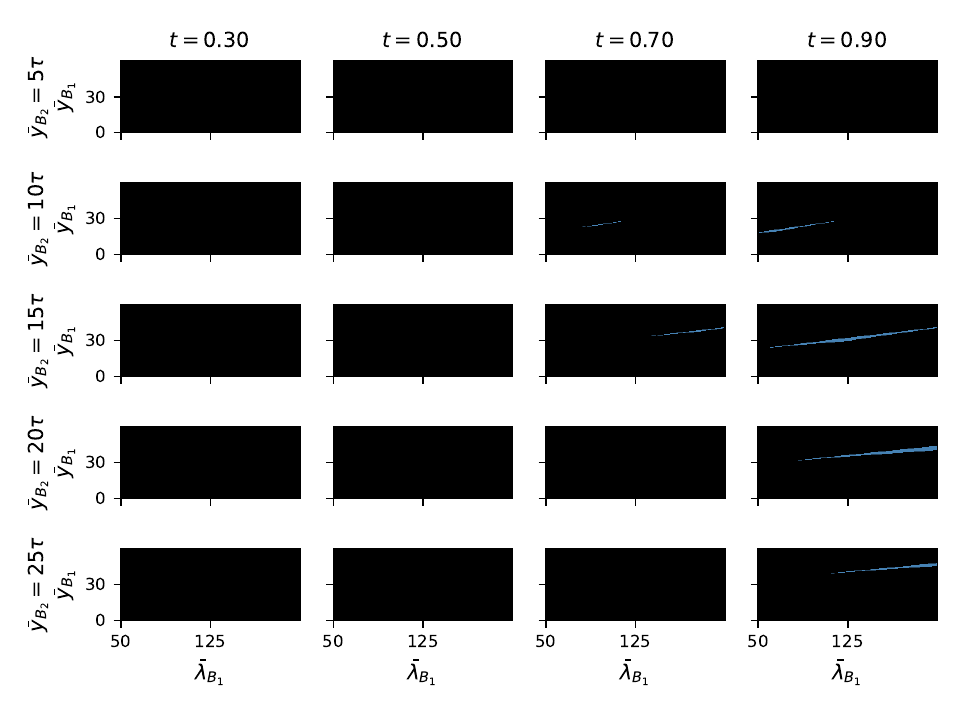}
    \caption{Blue points are points $(y_B, \lambda_B)$ such that $\chi(B_1) \geq \chi(B_2) \geq \chi(B)$. They are the only set of points where D-PAV could theoretically fail as it does not consider merging $B_1$ and $B_2$ when $\chi(B_1) \geq \chi(B_2)$.}
    \label{fig:zone_merging}
\end{figure}

\bibliographystyle{IEEEtran}
\bibliography{references.bib}

\vfill
\newpage
\setcounter{page}{1}
\setcounter{section}{0} %
\renewcommand{\appendixname}{Supplementary}
\onecolumn
\begin{center}
\huge \thetitle \\ Supplementary Material  \\[0.5cm]
\large Anne Gagneux, Mathurin Massias and Emmanuel Soubies \\
\end{center}

\begin{multicols}{2}
 \begin{figure}
\caption{tt}
\end{figure}
\section{Properties of Sorted Penalties}
\label{appendix:properties_sorted}

We recall several properties of proximal operators of sorted penalties, that will ease their computation, and give the proof of \Cref{prop:comput_prox,prop:prox_cone}.
Note that these properties have already been shown in several previous works \cite{zeng2014ordered,bogdan2015slope,feng2019sorted}. First, to ease the presentation of the following proofs,  let introduce the objective function of the sorted proximal operator
\begin{align}
 Q_{\eta \varPsi}(\rvx; \rvy) & := \frac{1}{2\eta} \|\rvy - \rvx\|_2^2 + \varPsi(\rvx) \\
 &  = \frac{1}{2\eta} \|\rvy - \rvx\|_2^2 + \sum_{i=1}^p \psi(|x_{(i)}|,\lambda_i),
\end{align}
which must not be confused with the non-sorted objective $P_{\eta \psi}$ defined in~\Cref{eq:unsort_prox_obj}.

\begin{proposition}[Sign of the proximal operator]
    \label{prop:sign}
    Let $\rvx \in \sR^p$ and $\rvp \in \prox_{\eta \varPsi}(\rvx)$.
    Then, for every $i \in \intervalle{1}{p} $:
    \begin{equation*}
        p_i x_i \geq 0 \ .
    \end{equation*}
\end{proposition}
\begin{proof}
    Suppose that there exists $i \in \llbracket 1, p \rrbracket $ such that $x_i p_i < 0$.
    Since for every $\rvx \in \sR^p$, $\varPsi(\rvx) = \varPsi(|\rvx|)$, we can assume $x_i > 0$ and $p_i < 0$ without loss of generality.
    It then follows that:
    \begin{equation*}\label{eq:proof_sign_prox-2}
        (x_i - p_i)^2 > (x_i- |p_i|)^2.
    \end{equation*}
    Let $\mathbf {\tilde p}$ be equal to $\rvp$ everywhere except at index $i$ where $\tilde p_i = |p_i|$.
    One has $\varPsi(\mathbf{\tilde p}) = \varPsi(\rvp)$ and:
    \begin{align*}
        Q_{\eta \varPsi} (\mathbf{\tilde p}, \rvx) & = \frac{1}{2} \sum_{j\neq i} (x_j - p_j)^2 + \frac{1}{2}(x_i - |p_i|)^2 + \eta \varPsi(\mathbf {\tilde p}) \\
       & < \frac{1}{2} \sum_{j\neq i} (x_j - p_j)^2 + \frac{1}{2}(x_i - p_i)^2 + \eta \varPsi(\mathbf{\tilde p}) \\
       & = \frac{1}{2} \sum_{j\neq i} (x_j - p_j)^2 + \frac{1}{2}(x_i - p_i)^2 + \eta \varPsi(\rvp)  \\
       &= Q_{\eta \varPsi} (\rvp, \rvx) \, ,
    \end{align*}
    which contradicts the optimality of $\rvp$.
\end{proof}
\begin{proposition}[Monotonicity of the proximal operator]
    \label{prop:mono}
    Let $\rvx \in \sR^p$ and $\rvp \in \prox_{\eta \varPsi}(\rvx)$. Then, the components of $\rvx$ and those of $\rvp$ are ordered in the same way: $\forall i, j \in \intervalle{1}{p}\ , x_i < x_j \implies p_i \leq p_j$.
\end{proposition}
\begin{proof}
    By contradiction, let $i, j \in \llbracket 1, p \rrbracket$ such that $x_i < x_j $ and $p_i > p_j$.
    Then:
    \begin{equation*}
        \label{ineq:rearrang1}
        (x_i- p_i)^2 + (x_j- p_j)^2 = x_i^2 + x_j^2 + p_i^2 + p_j^2 - 2 x_i p_i - 2 x_j p_j.
    \end{equation*}
    Yet by assumption $(x_i - x_j)(p_j - p_i) > 0$, and so $- x_i p_i - x_j p_j > - x_j p_i - x_i p_j$.
    It follows:
    \begin{align*}
        (x_i- p_i)^2 + (x_j- p_j)^2 & > x_i^2 + x_j^2 + p_i^2 + p_j^2 - 2 x_j p_i - 2 x_i p_j \\
        & = (x_j - p_i)^2 + (x_i - p_j)^2.
    \end{align*}
    Let $\mathbf{\tilde p}$ be equal to $\mathbf{p}$, except for $\tilde p_i = p_j$ and $\tilde p_j = p_i$.
    As $\varPsi$ is a sorted penalty, $\varPsi(\rvp) = \varPsi(\mathbf{\tilde p})$.
    Therefore,
    \begin{align*}
        & Q_{\eta \varPsi} (\mathbf{\tilde p}, \rvx) \\
        & = \frac{1}{2} \sum_{k \neq i,j} (x_k - p_k)^2 +  \frac{1}{2}(x_j - p_i)^2 +  \frac{1}{2}(x_i - p_j)^2  + \eta \varPsi(\mathbf {\tilde p}) \\
       & < \frac{1}{2} \sum_{k \neq i,j} (x_k - p_k)^2 + \frac{1}{2}(x_i - p_i)^2 + \frac{1}{2}(x_j - p_j)^2  + \eta \varPsi(\rvp) \\
       &= Q_{\eta \varPsi} (\rvp, \rvx) \, ,
    \end{align*}
    which contradicts the optimality of $\rvp$.
\end{proof}
\begin{proof}[Proof of \Cref{prop:comput_prox}]

    For any $\rvy \in \sR^p$, we know from \Cref{prop:sign} that $\rvy$ and $\prox_{\eta \varPsi}(\rvy)$ share the same sign element-wise.
    As flipping the sign of an element of $\rvy$ does not change the value of $ \varPsi( \rvy )$, it follows easily that
    $\prox_{\eta \varPsi}(\rvy) = \sign(\rvy) \odot \prox_{\eta \varPsi}(|\rvy|)$.
    It remains to show that for $\rvy$ non-negative:
    \begin{equation}\label{eq:prox_sorted_positive}
        \prox_{\eta \varPsi}(\rvy) =  \mathbf P_{|\rvy|}^\top \prox_{\eta \varPsi}(\mathbf P_{|\rvy|} \rvy) \, .
    \end{equation}
    For any permutation matrix $\mathbf Q$ and any $\rvx, \rvy \in \sR^p$, $\Vert \mathbf Q(\rvy-\rvx) \Vert^2 = \Vert \rvy-\rvx \Vert^2$.
    Moreover, by definition of the sorted penalty $\varPsi$, we have $\varPsi(\mathbf P_{|\rvy|} \rvy) = \varPsi(\rvy)$.
    Thus for any $\rvx, \rvy \in \sR^p$:
    \begin{equation*}
        \varPsi(\mathbf P_{|\rvy|} \rvy) + \frac{1}{2} \Vert \mathbf P_{|\rvy|}(\rvy-  \rvx) \Vert^2 = \varPsi(\rvy) + \frac{1}{2} \Vert \rvy- \rvx \Vert^2 \ ,
    \end{equation*}
    from which we deduce \Cref{eq:prox_sorted_positive}, by change of variable $\rvy' = \mathbf P_{|\rvy|} \rvy$.
\end{proof}
\begin{proof}[Proof of \Cref{prop:prox_cone}]
    First, let $\rvx^* \in \prox_{\eta \varPsi}(\rvy)$.
    From \Cref{prop:sign,prop:mono}, because $\rvy$ is sorted and has positive entries, so is $\rvx^*$: $\rvx^* = |\rvx^*|_\downarrow$.
    Thus,
    \begin{equation*}
        \prox_{\eta \varPsi}(\rvy) = \argmin_{\rvx \in \gK_p^+} Q_{\eta \varPsi}(\rvx, \rvy).
    \end{equation*}
    Then, the fact that for $\rvx  \in \gK_p^+$, $Q_{\eta \varPsi}(\rvx, \rvy) = P_{\eta \varPsi}(\rvx, \rvy)$ completes the proof.
\end{proof}
\section{Proofs} \label{appendix:proofs}

We first recall some important notations.
If we consider a block of consecutive indices $B = \intervalle{q}{r}$, We define as $F_B$ or $F_{q:r}$ the prox objective function on a block of successive indices $B$.
The prox objective function is defined on $\sR_+^p$ or on $\sR_+$, depending if it is evaluated at a vector or scalar point:
\begin{align*}
    F_B(\mathbf z) & = \sum_{i\in B} \frac12 ( y_i - z_i)^2 +    \lambda_i \psi_0(z_i)\ , \\
    F_B(z) & =   \sum_{i\in B} \frac12 ( y_i - z)^2 +    \lambda_i \psi_0(z)\ .
\end{align*}
In similar fashion, we denote $f_B$ or $f_{q:r}$ the derivative of $F_B$ on $\sR_{+*}^p$ or on $\sR_{+*}$.
In order to prove \Cref{theorem:ell_q_local}, we need the following lemma which shows that by merging two blocks that are not sorted in the correct order, the new value on the merged block is in between the two previous ones.

\begin{lemma}
    \label{lemma:merge}
    Let $B_1$ and $B_2$ be two consecutive blocks of indices. Then, if $\chi(B_1) < \chi(B_2)$, the merged block satisfies the following inequality.
    \begin{equation*}
        \chi(B_1) \leq \chi(B_1 \cup B_2) \leq \chi(B_2) \ .
    \end{equation*}
\end{lemma}

\begin{proof}
    We denote $B := B_1 \cup B_2$ the merged block.
    We first study the case where $\chi(B_1) >0$, implying $\chi(B_1) = \rho^+(\bar y_{B_1},  \bar \lambda_{B_1})$.
    It holds
    \begin{align}
        f_B(\chi(B_1)) & = f_{B_1}(\chi(B_1)) + f_{B_2} (\chi(B_1)) \\
        & =  f_{B_2} (\chi(B_1)) < 0 ,
    \end{align}
    where the last inequality comes from $\chi(B_2) > \chi(B_1) \geq m( \bar \lambda_{B_1})> m ( \bar \lambda_{B_2})$.
    Likewise,
    \begin{align}
        f_B(\chi(B_2)) & = f_{B_1}(\chi(B_2)) + f_{B_2} (\chi(B_2)) \\
        & =  f_{B_1} (\chi(B_2)) >0
    \end{align}
    where the last inequality comes from $\chi(B_2) > \chi(B_1)$.
    So, the root of $f_{B}$ lies in $[\chi(B_1), \chi(B_2)]$.

    Now, in the case where $\chi(B_1) = 0$, it implies $F_{B_1}$ is increasing on $\sR_+$, so $f_{B_1}$ is always positive.
    If, $\chi(B) = 0$, the result is trivial, otherwise it holds
    \begin{align}
        0 &= f_{B}(\chi(B)) \\
        & = f_{B_1}(\chi(B)) + f_{B_2}(\chi(B)) \\
        & > f_{B_2}(\chi(B)).
    \end{align}
    So, $\chi(B) \in [m(\lambda_{B_2}), \chi(B_2)]$ and  we recover the result.
  \end{proof}

\begin{proof}[Proof of \Cref{theorem:ell_q_local}]
    \textbf{Necessary conditions}
    Let $\rvu$ be a local minimizer of $(P_p)$, \textit{i.e.} $\rvu \in \gL_p$ and let  $B = \intervalle{q}{r}$ be a maximal block of indices on which $\rvu$ is constant equal to  $\tilde u$, i.e. $u_{q-1}> u_q = \dots = u_r = \tilde u > u_{r+1} $.
    We demonstrate the two points of the theorem:
    \begin{enumerate}[label=(\roman*)]
        \item %
        If $\tilde u \not \in \{ \chi(B), 0\}$, then $\tilde u$ is not a local minimizer of $F_B$ (\Cref{prop:variation_scalar}).
        One can thus either infinitesimally increase or decrease $\tilde u$ so as to decrease $F_B$ while letting $\rvu$ feasible.
        Then, the total objective value $F_{1:p}(\rvu)= F_{1:q-1}(\rvu_{1:q-1}) + F_B (\tilde u) + F_{r+1:p}(\rvu_{r+1:p})$ decreases which contradicts $\rvu$ being a local minima.
        \item
         Now, let $\tilde u = \chi(B)$ and pick an index $j \in \intervalle{q}{r-1}$. We distinguish two cases
        \begin{itemize}[leftmargin=*]
        	\item If $\tilde u = \chi(B) = 0$, then if $\chi(\intervalle{q}{j})=0$ and/or $\chi(\intervalle{j+1}{r}) = 0$, either \Cref{theorem:ell_q_local-1} or~\Cref{theorem:ell_q_local-2} hold true. It remains to prove that one of these two inequalities is also valid when  $\chi(\intervalle{q}{j})>0$ and $\chi(\intervalle{j+1}{r}) > 0$. In this case we necessarily have $\chi(\intervalle{q}{j}) \geq \chi(\intervalle{j+1}{r})$, otherwise \Cref{lemma:merge} would contradicts the fact that  $\chi(B) = 0$. Hence~\Cref{theorem:ell_q_local-2} holds true.
         \item If $\tilde u = \chi(B)>0$, assume that  $F_{q:j}$ is non-increasing at $\tilde{u}$. Then, given that $F_{q:j}$ admits a finite number of critical points minimizers (at most 3), there exists a sufficiently small $\zeta>0$ such that, for all $\epsilon \in (0,\zeta)$,
        \begin{equation*}
        	\tilde{u} + \epsilon < u_{q-1} \quad \text{ and } \quad
        	 F_{q:j}(\tilde{u} + \epsilon) < F_{q:j}(\tilde{u}),
		\end{equation*}
		which implies that the point $\vv$ defined as $\rvu$ everywhere except on the bloc $\intervalle{q}{j}$ where it takes the value 	$\tilde{u} + \epsilon$ is feasible and verifies $F(\vv) < F(\rvu)$. This contradicts the local optimality of $\rvu$. As such $F_{q:j}$ is increasing at $\tilde{u}$. Similarly, we can show that $F_{j+1:r}$ is decreasing at $\tilde{u}$. Then, we get from the possible variations of $F_B$ for a block $B$ that
		\begin{align*}
		& \tilde{u} \leq \chi(\intervalle{j+1}{r}),  \  \text{{$\tilde{u}$ on a $\searrow$ part of $ F_{j+1:r}$}}\\
		& \tilde{u} \leq \chi(\intervalle{q}{j}) \text{ or } \tilde{u} \geq \chi(\intervalle{q}{j}),  \ \text{{$\tilde{u}$ on an $\nearrow$ part of $ F_{q:j}$}}
		\end{align*}
		Moreover, when both $ \tilde{u} \leq \chi(\intervalle{j+1}{r}) $ and $\tilde{u} \leq \chi(\intervalle{q}{j}) $, we get from $\tilde{u}>0$ together with \Cref{lemma:merge} that $\chi(\intervalle{q}{j}) \geq \chi(\intervalle{j+1}{r}) \geq \tilde{u}$, which completes the proof.
		 \end{itemize}
    \end{enumerate}

    \textbf{Sufficient conditions} Let $\rvu$ be such that the conditions of \Cref{theorem:ell_q_local} hold. First of all, note that only the last block of $\rvu$ can be $0$. Then, for each  block $B=\intervalle{q}{r}$ of $\rvu$ such that its value $\chi(B) >0$, let define for all $j \in \intervalle{q}{r-1}$, $ \zeta_{B,j}>0$ such that $F_{q:j}$ is increasing on $(\chi(B), \chi(B)+ \zeta_{B,j})$ and $F_{j+1:r}$ is decreasing on $(\chi(B) -  \zeta_{B,j}, \chi(B))$. Note that such  $\zeta_{B,j}$ always  exists from the conditions of \Cref{theorem:ell_q_local}.\\
    We define $\bar \zeta = \min_{B, j \in B} \zeta_{B,j}$. Now, take $\vv \in \mathcal{B}(\mathbf{0},\bar{\zeta})$ such that $\rvu + \vv$ is a feasible point. Hence, on each block $B$ of $\rvu$, the values of $\vv$ are sorted in decreasing order (to preserve feasibility). Let us look at one block of $\rvu$, say the first one  $B=\intervalle{1}{r}$ (wlog), for which we denote by $l \in \intervalle{1}{r}$ be the index corresponding to the first negative value of $\vv$. We have the decomposition $\vv_B = \sum_{j=1}^r \mathbf{w}^{j}$ where
    \begin{align}
    	&\textstyle \mathbf{w}^{j} = (v_{j} - \max(v_{j+1},0)) \mathbf{1}_{1:j}, \forall j \in \intervalle{1}{l-1} \\
    	&\textstyle \mathbf{w}^{j} = (v_{j} - \min(v_{j-1},0)) \mathbf{1}_{j:r}, \forall j \in \intervalle{l}{r}
    \end{align}
    with $\mathbf{1}_{m:n}$ a vector of size $r$ with ones between $m$ and $n$. See \Cref{figure:schema_preuve_local_ell_q} for an illustration.
    Moreover, for all $j \in B$, we have  $\mathbf{w}^{j} \in \mathcal{B}(\mathbf{0},\bar{\zeta})$. It then follows from the fact that, on $\mathcal{B}(\mathbf{0},\bar{\zeta})$, $F_{1:j}$ is increasing $ \forall j \in \intervalle{1}{l-1}$  and  $F_{j:r}$ is decreasing $\forall j \in \intervalle{l}{r}$ (by definition of $\bar{\zeta}$), that
    \begin{align}
        F_B(\rvu_B) & \leq F_B(\rvu_B + \mathbf{w}^{1}) \\
        & \leq F_B(\rvu_B + \mathbf{w}^{1} +  \mathbf{w}^{2}) \\
        & \leq F_B(\rvu_B + \sum_{j=1}^r \mathbf{w}^{j}) \\
        & = F_B(\rvu_B + \vv_B).
    \end{align}

    Repeating this for each block of $\rvu$ and observing that if its last block is $0$, then it is on an increasing part of each sub-function over this block, we get that $F(\rvu) \leq F(\vv)$, for all $\vv \in \mathcal{B}(\mathbf{0},\bar{\zeta})$ such that $\rvu + \vv$ is feasible. This proves that $\rvu$ is a local minimizer.
\end{proof}

\begin{figure*}
    \centering
    \includegraphics[scale=0.7]{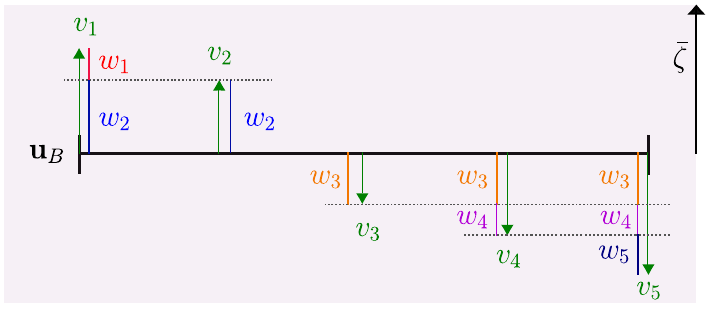}
    \caption{Illustration of the proof of \Cref{theorem:ell_q_local}. }
    \label{figure:schema_preuve_local_ell_q}
\end{figure*}

\begin{lemma} \label{lem:bloc_above_max_subbloc}
    Let $B_1$ and $B_2$ be two successive blocks of indices.
    Then $$\chi(B_1 \cup B_2) \leq \max\left(\chi(B_1),\chi(B_2) \right).$$
    \end{lemma}

    \begin{proof}Let $B = B_1 \cup B_2$ .
    If $\max\left(\chi(B_1),\chi(B_2) \right) = 0$, then $\chi(B) =0$ (as in this case $F_{B_1}$ and $F_{B_2}$ are convex).
    Now, assume $\max\left(\chi(B_1),\chi(B_2) \right) > 0$.
    If $\chi(B)=0$, then the result trivially holds.
    If $\chi(B)>0$, assume that the statement is not true with $\chi(B) > \max\left(\chi(B_1),\chi(B_2) \right) >0$. Hence, from the variations of $F_{B_1}$ and $F_{B_2}$, we get that $\chi(B)$ belongs to an increasing part of each of these two functions. Hence, there exists $\varepsilon_0>0$ such that $\forall \varepsilon \in (0,\varepsilon_0)$,
    \begin{align*}
        F_B(\chi(B) - \varepsilon) & = F_{B_1}(\chi(B) - \varepsilon) +  F_{B_2}(\chi(B) - \varepsilon) \\
        & \leq F_{B_1}(\chi(B)) +  F_{B_2}(\chi(B)) = F_B( \chi(B)),
    \end{align*}
    which contradicts the fact that  $\chi(B)>0$ is the unique positive minimizer of $F_B$.
    \end{proof}

\begin{lemma}
    \label{lemma:pav_blocks}
Consider $\rvu \in \gK_p^+$ the output of the PAV algorithm.
    On each maximal block of $B = \intervalle{q}{r}$ where $\rvu$ is constant equal to $\tilde u$, it holds for all $q \leq j < r$,
    \begin{equation}
        \chi(\intervalle{q}{j}) \leq   \tilde u  \leq \chi(\intervalle{j+1}{r}).
    \end{equation}
\end{lemma}
\begin{proof}
Figure~\ref{fig:proof_PAVlocmin-1} provides a scheme that illustrate the main ideas of this proof.
    Let $B = \intervalle{q}{r}$ be a block of $\rvu$.
    If $\chi(B) =0$, implying it is the last block of $\rvu$, the second point of \Cref{theorem:ell_q_local} is immediately satisfied.
    From now on, $\chi(B)>0$. Assume that this block has been involved at most within one backward pass. Then, given the steps of the PAV algorithms, there exists a partition of  $B$ in $I$ blocks defined by the indices $q = q_1 < q_2 < \cdots < q_I \leq r$ (where $q_i$, is the first index of the i-th block of the partition and we set $q_{I+1}=r+1$ by convention) such that $\forall i \in \intervalle{1}{I}$,
    \begin{itemize}
       \item $\forall  j\in \intervalle{q_i+1}{q_{i+1}-1}$, the singleton $\{j\}$ was added during a forward pass, i.e.,
       \begin{equation}\label{eq:proof_PAVlocmin-1}
        \chi(\intervalle{q_i}{j-1}) \leq \chi(\intervalle{q_i}{j}) \leq \chi(\{j\}) \ .
       \end{equation}
       \item The block $ \chi(\intervalle{q_i}{q_{i+1}-1})$ was merged during a backward pass
       \begin{equation}\label{eq:proof_PAVlocmin-2}
       \chi(\intervalle{q_{i}}{q_{i+1}-1}) \leq \chi(\intervalle{q_{i}}{r}) \leq \chi(\intervalle{q_{i+1}}{r}) \ .
       \end{equation}
   \end{itemize}
   where in both cases the inequalities with the central term are due to  \Cref{lemma:merge}.\\
   Now let take $l \in B \setminus \{r\}$ and denote by $i^*$ the index of the block such that $l \in \intervalle{q_{i^\star}}{q_{i^\star +1}-1}$.
     We have from~\Cref{eq:proof_PAVlocmin-1} together with the feasibility of the forward pass,
       \begin{align}\label{eq:proof_PAVlocmin-2bis}
            \chi(\intervalle{q_1}{q_2-1}) > \cdots > \chi(\intervalle{q_{i^*}}{q_{i^*+1}-1})
            \underset{(\ref{eq:proof_PAVlocmin-1})}{\geq}  \chi(\intervalle{q_{i^*}}{l}).
      \end{align}
    Then, we have
      \begin{align}
        \chi(\intervalle{q_1}{l}) & \leq \max\left( \max_{i=1:i^*-1}  \chi(\intervalle{q_{i}}{q_{i+1}-1}),  \chi(\intervalle{q_{i^*}}{l})\right) \notag \\
        & \underset{\eqref{eq:proof_PAVlocmin-2bis}}{\leq} \chi(\intervalle{q_{1}}{q_2-1}) \underset{(\ref{eq:proof_PAVlocmin-2})}{\leq} \chi(\intervalle{q_{1}}{r}) = \chi(B) .\label{eq:proof_PAVlocmin-3}
      \end{align}
            where the first inequality is due to a recursive application of \Cref{lem:bloc_above_max_subbloc}.
    From the two previous equations, we have
      \begin{equation}\label{eq:proof_PAVlocmin-4}
            \chi(\intervalle{q_{i^*}}{l}) <  \chi(B) \ .
      \end{equation}
      Now, combining the inequalities in~\Cref{eq:proof_PAVlocmin-2}, we deduce
      \begin{equation}\label{eq:proof_PAVlocmin-5}
          \chi(\intervalle{q_{i^*}}{r}) \geq  \chi(\intervalle{q_{1}}{r}) = \chi(B) \ .
      \end{equation}
    \Cref{lem:bloc_above_max_subbloc} states that
      \begin{equation}
        \chi(\intervalle{q_{i^\star}}{r}) \leq \max(\chi(\intervalle{q_{i^\star}}{l}),\chi(\intervalle{l+1}{r}) ).
      \end{equation}
      Since, from \Cref{eq:proof_PAVlocmin-4,eq:proof_PAVlocmin-5}, $\chi(\intervalle{q_{i^\star}}{r}) >\chi(\intervalle{q_{i^\star}}{l})$, we get
      \begin{equation}\label{eq:proof_PAVlocmin-6}
             \chi(\intervalle{l+1}{r}) \geq  \chi(\intervalle{q_{i^*}}{r}) \geq   \chi(B) \ .
      \end{equation}
      Finally, we conclude that the block $B$ satisfies
      \begin{align*}
        \chi(\intervalle{q_1}{l})   \underset{\eqref{eq:proof_PAVlocmin-3}}{\leq} \chi(B) \underset{\eqref{eq:proof_PAVlocmin-6}}{\leq} \chi(\intervalle{l+1}{r}).
      \end{align*}
       If $B$ was the result of multiple backward passes (not at most one as assumed above), we can get the results through the repetition of the above arguments.
\end{proof}

   \begin{figure*}
   \centering
   \includegraphics[scale=0.6]{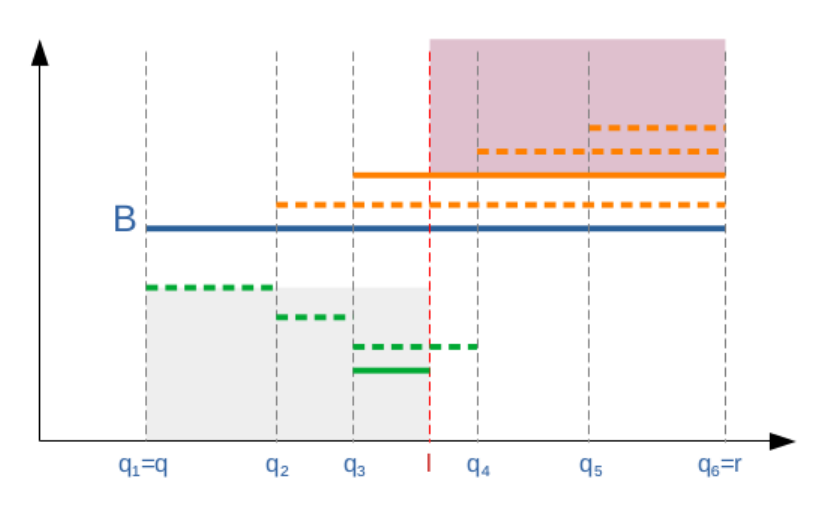}
   \caption{Illustration of the proof of \Cref{lemma:pav_blocks}. Dashed green blocks illustrate the feasibility of the forward pass while the continuous green block is due to~\Cref{eq:proof_PAVlocmin-1}. Together, they illustrate the inequalities in~\Cref{eq:proof_PAVlocmin-2bis}. Then, inequalities in~\Cref{eq:proof_PAVlocmin-3} impose that $\chi(\intervalle{q}{l})$ lies in the gray area. On the other hand, dashed orange blocs illustrate the constraints in~\Cref{eq:proof_PAVlocmin-2} due to the backward pass (with in particular inequalities in~\Cref{eq:proof_PAVlocmin-5}). Finally, the red area correspond to~\Cref{eq:proof_PAVlocmin-6}, highlighting the region were $ \chi(\intervalle{l}{r})$ belongs to. We clearly see that $l$ separate the block in two subblocks satisfying the conditions of \Cref{theorem:ell_q_local}.\label{fig:proof_PAVlocmin-1}  }
   \end{figure*}

\begin{proof}[Proof of \Cref{thm:PAV_loc_cv}]
    Let $\rvu$ be the solution returned by the PAV algorithm.
    First of all, by construction of the PAV algorithm, each block of $\rvu$ satisfies the first point of~\Cref{theorem:ell_q_local}.
    The previous \Cref{lemma:pav_blocks} proves that the second point of the Theorem is also verified, as \Cref{theorem:ell_q_local-1} holds.
   \end{proof}

\begin{lemma}
    \label{lemma:merging}
    Let the regularization parameters $(\lambda_i)_i$ be such that \Cref{hypothesis:global_mins} is satisfied.
    For any successive blocks of indices $B_1$ and $B_2$ such that $\chi(B_1) > \chi(B_2)$, we cannot have  $\chi(B_1) >  \chi(B_2) > \chi(B_1 \cup B_2) $ with $\chi(B_1 \cup B_2) > 0 $ global minimizer of $F_{B_1\cup B_2}$.
\end{lemma}

\begin{proof}
    We denote $B  := B_1 \cup B_2$.
    By contradiction, let assume $\chi(B_1)  > \chi(B_2) >  \chi(B)$ with $\chi(B) >0$ global minimizer of $F_B$.
    Then, by \Cref{prop:variation_scalar}, we have the following inequality:
    \begin{equation*}
        \chi(B)  = \rho^+(\bar y_B, \bar \lambda_B) \geq m(\bar \lambda_B) \ .
    \end{equation*}
    Besides, $f_{B_1}(\chi(B)) > 0$, because $f_B(\chi(B)) = f_{B_1} (\chi(B)) + f_{B_2}(\chi(B)) = 0$ with $ f_{B_2}(\chi(B)) < 0$ (because $\chi(B_2)> \chi(B) \geq m(\bar \lambda_B) \geq m(\bar \lambda_{B_2})$). We deduce $\chi(B) < m(\bar \lambda_{B_1})$ (actually, $\chi(B)$ is smaller than the smallest root of $f_{B_1}$ corresponding to the local maxima of $F_{B_1}$).
    Yet, as by assumption $\chi(B)> 0$ and $\chi(B)$ global minimizer of $F_B$, it implies $\bar y_B > T(\bar \lambda_B)$.
    As $\rho^+$ is non-decreasing with $y$, it implies that
    \begin{align}
        m(\bar \lambda_{B_1}) > \chi(B)  & = \rho^+(\bar y_B, \bar \lambda_B) \\
        & \geq \rho^+(T(\bar \lambda_B), \bar \lambda_B),
    \end{align}
    which contradicts \Cref{hypothesis:global_mins} and ends the proof.
\end{proof}

\begin{proof}[Proof of \Cref{theorem:oscar_ell_q_global_nec}]
    \label{proof:theorem:oscar_ell_q_global_nec}
    Let $\rvx$ be a global minimizer of $(P_p)$ and $i$ be the first index of its last block, i.e. such that $x_i  = \dots = x_p = \tilde x$ and $x_{i-1} > \tilde x$. We distinguish two cases depending on the value of $\tilde{x}$.
    \begin{itemize}
        \item If $\tilde x = 0$, then $\rvx_{[i-1]}$ must be a global minimizer of $(P_{i-1})$, otherwise this would contradict the optimality of $\rvx$ for $(P_p)$.
        As such,  we have $\rvx_{[i-1]} \in \gL_{i-1}$, and $\rvx \in \gS_p$.

        \item If $\tilde x = \chi(\intervalle{i}{p}) > 0$.
        Then, $\tilde x$ is the global  minimizer of $F_{i:p}$ (otherwise $(\rvx_{[i-1]}, 0, \dots, 0)$ would lead to a better objective value).
        Now, as in the statement of the theorem, let $i^\star$ denotes the largest index such that  there exists $\rvu^{i^\star-1} \in \gL_{i^\star-1}$ with $\mathbf{v} = (\rvu^{i^\star-1}, \chi(\intervalle{i^\star}{p}),\dots, \chi(\intervalle{i^\star}{p} ) )$  feasible and assume that $i < i^*$. Clearly $\mathbf{v} \in \gL_{p}$  otherwise this would contradict $\rvu^{i^\star-1} \in \gL_{i^\star-1}$.
        Given that $\rvx$ is a global (and thus local) minimizer of $(P_p)$, we get from \Cref{theorem:ell_q_local}, that one of the following two statements holds true
        \begin{itemize}
            \item $\chi(\intervalle{i}{i^*-1}) \leq  \tilde x \leq \chi(\intervalle{i^*}{p})$,
            \item $\chi(\intervalle{i}{i^*-1}) \geq   \chi(\intervalle{i^*}{p}) \geq \tilde x$.
        \end{itemize}
        Yet, \Cref{lemma:merging} eliminates the case $\chi(\intervalle{i}{i^*-1}) \geq   \chi(\intervalle{i^*}{p}) \geq \tilde x$ (because $\tilde x$ is the global minimizer of $F_{i:p}$), so we must have $\chi(\intervalle{i}{i^*-1}) \leq  \tilde x \leq \chi(\intervalle{i^*}{p})$. \\
        By construction, $\rvu^{i^*-1} \in \gL_{i^*-1}$ and $u^{i^*-1}_{i^*-1} \geq \argmin_{x \in \sR_+} F_{i^*:p}(x)$.
        \begin{itemize}
            \item \textsc{Case $\argmin_{x \in \sR_+} F_{i^*:p}(x) = \chi(\intervalle{i^*}{p})$.}
            There exists a block $B$ of $\rvu^{i^*-1}$ which contains $i$ and we denote it $B = \intervalle{s}{t}$ with $s \leq i \leq t < i^*$.
            Then by feasibility of $(\rvu^{i^*-1}, \chi(\intervalle{i^*}{p}),\dots, \chi(\intervalle{i^*}{p} ) )$, we have  $ \chi(\intervalle{i^*}{p}) < \chi(\intervalle{s}{t}) $.
            Moreover, from the same arguments as used previously, involving \Cref{theorem:ell_q_local} and \Cref{lemma:merging}, we get $\chi(\intervalle{i}{t}) \leq \tilde x$.
            Finally, from the local optimality of $(\rvu^{i^*-1}, \chi(\intervalle{i^*}{p}),\dots, \chi(\intervalle{i^*}{p} ) )$, we obtain from \Cref{theorem:ell_q_local}  that $\chi(\intervalle{i}{t}) \geq \chi(\intervalle{s}{t})$. Gathering these different inequalities, we get
            \begin{equation*}
                \chi(\intervalle{i^*}{p}) < \chi(\intervalle{s}{t}) \leq  \chi(\intervalle{i}{t}) \leq \tilde x \ ,
            \end{equation*}
            which leads to a contradiction with $\tilde x$ being smaller than $ \chi(\intervalle{i^*}{p})$.
            \item \textsc{Case $\argmin_{x \in \sR_+} F_{i^*:p}(x) = 0$.}  It contradicts $\rvx$ being a global minimizer of $(P_p)$, as one could set $\rvx_{[i^*:p]}$ to $0$ and get a better global objective value.
        \end{itemize}
        To sum up, we cannot have $i < i^*$. Then, if $i> i^*$, it means that $\rvx_{[1:i-1]} \not \in \gL_{i-1}$ (by definition of $i^\star$), so one can change some components while remaining feasible and decrease the objective value which contradicts $\rvx$ being a global minimizer. We finally get that $i$ must equal $i^*$ so $\rvx \in \gS_p$. \qedhere
    \end{itemize}
\end{proof}

\begin{proof}[Proof of \Cref{example:extension_concave}]
    \label{proof:concave_ellq}
    We want to show that for any block $B = \intervalle{l}{r}$, any sub-block $B_1 = \intervalle{l}{r'}$, we have
    \begin{equation}
        \rho^+(T(\bar \lambda_B), \bar \lambda_B) \geq m(\bar \lambda_{B_1}).
    \end{equation}
    We first compute $\rho^+(T(\bar \lambda_B), \bar \lambda_B)$.
    It corresponds to the $0$ of
    \begin{equation}
        f : z \mapsto z - T(\bar \lambda_B) + \bar \lambda_B q  z^{q-1} .
    \end{equation}
    Using the expression of $T(\bar \lambda_B)$, $f$ (the derivative of the scalar proximal objective) rewrites as
    \begin{equation}
        f(z) = z - \frac{1}{2}\frac{2-q}{1-q}(2 \bar \lambda_B (1-q))^{\frac{1}{2-q}}  +  \bar \lambda_B q  z^{q-1} .
    \end{equation}
    Consider $M := (2 \bar \lambda_B (1-q))^{\frac{1}{2-q}}$.
    Then,
    \begin{align}
        f(M) & = M - T(\bar \lambda_B) + \bar \lambda_B q M^{q-1} \\
        & = M (1 + \bar \lambda_B q M^{q-2} ) - T(\bar \lambda_B) \\
        & = M \left( \frac{1}{2} \frac{2-q}{ 1-q} \right) - T(\bar \lambda_B) =0
    \end{align}
    So, $\rho^+(T(\bar \lambda_B), \bar \lambda_B) = M$. Moreover, one can see that $m(\lambda) = (\lambda q (1-q))^{\frac{1}{2-q}}$ from~\eqref{eq:m_lamb}.
    Then, $\rho^+(T(\bar \lambda_B), \bar \lambda_B) > m(\bar \lambda_{B_1})$ is equivalent to $\bar \lambda_B > \bar \lambda_{B_1} \frac q 2$. We can verify that this inequality holds true under the assumption made on the sequence $(\lambda_i)_i$ in \Cref{example:extension_concave}.
    \begin{align*}
        \bar \lambda_B & = \frac{1}{\card{B}} \sum_{i=l}^r \Lambda(i) \\
         & =  \frac{1}{\card{B}} \sum_{i=l}^r \int_i^{i+1} \Lambda(i) dt \\
        & \geq  \frac{1}{\card{B}} \sum_{i=l}^r \int_i^{i+1} \Lambda(t) dt \mkern-20mu  & \mkern-20mu \text{($\Lambda$ non-increasing)} \\
        &  = \frac{1}{\card{B}} \int_{l}^{r+1} \Lambda(t) dt   \\
        & \geq \frac{\Lambda(l)+ \Lambda(r+1)}{2} \mkern-20mu & \mkern-20mu \text{(Hermite inequality)} \\
        & \geq  \frac{\Lambda(l)}{2} \mkern-20mu &  \mkern-20mu\text{($\Lambda$ non-negative)} \\
        & > \frac q 2 \Lambda(l) \mkern-20mu & \mkern-20mu (q<1)\\
        & \geq \frac q 2 \frac{1}{\card{B}} \sum_{i=l}^j \Lambda(i) \mkern-30mu & \mkern-30mu  \text{($\Lambda$ non-increasing)} \\
        & = \frac q 2 \bar \lambda_{B_1}
    \end{align*}
\end{proof}

 \end{multicols}
\end{document}